\documentclass[reqno,12pt]{amsart}
\usepackage{mathrsfs,amsmath,amssymb}
\usepackage{paralist}
\usepackage{tikz}
\usepackage{caption}
\usepackage{graphics} 
\usepackage{epsfig} 
\usepackage[colorlinks = true]{hyperref}
\hypersetup{urlcolor = red, citecolor = blue}
\usepackage[top = 1in,bottom = 1in,left = 1in,right = 1in]{geometry}
\usepackage{cite}

\usepackage{esint}

\usepackage[pagewise]{lineno}

\allowdisplaybreaks[4]

\newtheorem{thm}{Theorem}[section]

\newtheorem{lem}[thm]{Lemma}
\newtheorem{prop}[thm]{Proposition}
\theoremstyle{definition}

\newtheorem{rem}{Remark}[section]
\numberwithin{equation}{section}

\begin{document}
\title[Quasilinear two-species chemotaxis system with two chemicals]
{Critical blow-up lines in a two-species quasilinear chemotaxis system with two chemicals }

\author[Zeng]{Ziyue Zeng}%
\address{Z. Zeng: School of Mathematics, Southeast University, Nanjing 211189, P. R. China}
\email{ziyzzy@163.com}

\author[Li]{Yuxiang Li$^{\star}$}
\address{Y. Li: School of Mathematics, Southeast University, Nanjing 211189, P. R. China}
\email{lieyx@seu.edu.cn}

\subjclass[2020]{35B44, 35B33, 35K57, 35K59, 35Q92, 92C17.}%

\keywords{Quasilinear two-species chemotaxis system with two chemicals; critical blow-up lines; finite time blow up; global boundedness}
\thanks{$^{\star}$Corresponding author}
\thanks{Supported in part by National Natural Science Foundation of China (No. 12271092, No. 11671079) and the Jiangsu Provincial Scientific Research Center of Applied
Mathematics (No. BK20233002).}
\begin{abstract}
In this study, we explore the quasilinear two-species chemotaxis system with two chemicals
\begin{align}\tag{$\star$}
\begin{cases}
u_t = \nabla \cdot(D(u)\nabla u) -  \nabla \cdot \left(S(u) \nabla v\right), & x \in \Omega, \ t > 0, \\ 
0 = \Delta v - \mu_w + w, \quad \mu_w=\fint_{\Omega}w, & x \in \Omega, \ t > 0, \\ 
w_t = \Delta w -  \nabla \cdot \left(w \nabla z\right), & x \in \Omega, \ t > 0, \\ 
0 = \Delta z - \mu_u + u, \quad \mu_u=\fint_{\Omega}u, & x \in \Omega, \ t > 0, \\
\frac{\partial u}{\partial \nu} = \frac{\partial v}{\partial \nu} = \frac{\partial w}{\partial \nu} = \frac{\partial z}{\partial \nu} = 0, & x \in \partial \Omega, \ t > 0, \\
u(x, 0) = u_0(x), \quad w(x, 0) = w_0(x), & x \in \Omega,
\end{cases}
\end{align}
where $\Omega \subset \mathbb{R}^n$ ($n \geq3$) is a smooth bounded domain. The functions $D(s)$ and $S(s)$ exhibit asymptotic behavior of the form 
\begin{align*}
D(s) \simeq k_D s^p \ \text {and} \ S(s) \simeq k_S s^q, \quad s \gg 1
\end{align*}
with $p,q \in \mathbb{R}$. We prove that 
\begin{itemize}
 \item when $\Omega$ is a ball, if $q-p>2-\frac{n}{2}$ and $q>1-\frac{n}{2}$, there exist radially symmetric initial data $u_0$ and $w_0$, such that the corresponding solutions blow up in finite time;
 \item for any general smooth bounded domain $\Omega \subset \mathbb{R}^n$, if $q-p<2-\frac{n}{2}$, all solutions are globally bounded;
 \item for any general smooth bounded domain $\Omega \subset \mathbb{R}^n$, if $q<1-\frac{n}{2}$, all solutions are global.
\end{itemize}
We point out that our results implies that the system ($\star$) possess two critical lines $
q-p=2-\frac{n}{2}$ and $q=1-\frac{n}{2}$ to classify three dynamics among global boundedness, finite-time blow-up, and global existence of solutions to system ($\star$).

\end{abstract}
\maketitle

\section{Introduction}

In this study, we explore the two-species chemotaxis system with two chemicals proposed by Tao and Winkler \cite{2015-DCDSSB-TaoWinkler}. Since experimental evidence shows that the diffusion of the chemical substance is faster than the random motion of cells \cite{1992-TAMS-JaegerLuckhaus}, we consider the following quasilinear system with Jäger-Luckhaus form
\begin{align}\label{eq1.1.0}
\begin{cases}
u_t = \nabla \cdot(D(u) \nabla u) -  \nabla \cdot \left(S(u) \nabla v\right), & x \in \Omega, \ t > 0, \\ 
0 = \Delta v - \mu_w + w, \quad \mu_w=\fint_{\Omega}w, & x \in \Omega, \ t > 0, \\ 
w_t = \Delta w -  \nabla \cdot \left(w \nabla z\right), & x \in \Omega, \ t > 0, \\ 
0 = \Delta z - \mu_u + u, \quad \mu_u=\fint_{\Omega}u, & x \in \Omega, \ t > 0, \\
\frac{\partial u}{\partial \nu} = \frac{\partial v}{\partial \nu} = \frac{\partial w}{\partial \nu} = \frac{\partial z}{\partial \nu} = 0, & x \in \partial \Omega, \ t > 0, \\
u(x, 0) = u_0(x), \quad w(x, 0) = w_0(x), & x \in \Omega,
\end{cases}
\end{align}
where $\Omega\subset \mathbb{R}^n (n \geq 1$) is a smooth bounded domain and $D(s)$, $S(s)$ exhibit asymptotic behavior of the form 
\begin{align*}
D(s) \simeq k_D s^p \ \text {and} \ S(s) \simeq k_S s^q, \quad s \gg 1
\end{align*}
with $p,q \in \mathbb{R}$. Unlike the classical chemotaxis system, system (\ref{eq1.1.0}) exhibits a circular interaction structure: the density of the first species is denoted by $u(x,t)$, and its movement is influenced by the concentration of a chemical signal, denoted as $v(x,t)$, which is secreted by the second species. The density of the second species is represented by $w(x,t)$, and its individuals orient their movement along concentration gradients of a second signal with density
$z(x,t)$ which in turn is produced by the first species. 

Some scholars have investigated the two-species chemotaxis system with two chemicals
\begin{align}\label{eq1.1.0-1}
\begin{cases}
u_t = \nabla \cdot(D_1(u) \nabla u) -  \nabla \cdot \left(S_1(u) \nabla v\right), & x \in \Omega, \ t > 0, \\ 
0 = \Delta v - v + w,  & x \in \Omega, \ t > 0, \\ 
w_t = \nabla \cdot(D_2(w) \nabla w) -  \nabla \cdot \left(S_2(w) \nabla z\right), & x \in \Omega, \ t > 0, \\ 
0 = \Delta z - z + u,  & x \in \Omega, \ t > 0, \\
\frac{\partial u}{\partial \nu} = \frac{\partial v}{\partial \nu} = \frac{\partial w}{\partial \nu} = \frac{\partial z}{\partial \nu} = 0, & x \in \partial \Omega, \ t > 0, \\
u(x, 0) = u_0(x), \quad w(x, 0) = w_0(x), & x \in \Omega.
\end{cases}
\end{align}
Let $m_1:=\int_{\Omega} u_0dx+\int_{\Omega} w_0dx >0$. For the system (\ref{eq1.1.0-1}) with $ D_1(s)\equiv D_2(s)\equiv 1$ and $S_1(s)=S_2(s)=s$, Tao and Winkler \cite{2015-DCDSSB-TaoWinkler} found that 
\begin{itemize}
\item if either $n=2$ and $m_1$ lies below some threshold, or $n \geq 3$ and $\left\|u_0\right\|_{L^{\infty}(\Omega)}$ and $\left\|w_0\right\|_{L^{\infty}(\Omega)}$ are sufficiently small, solutions are globally bounded; 
\item if either $n=2$ and $m_1$ is suitably large, or $n \geq 3$ and $m_1$ is arbitrary, there exist initial data such that the corresponding solution blows up in finite time. 
\end{itemize}  
    Let $m_u:=\int_{\Omega} u_0 dx$ and $ m_w:=\int_{\Omega} w_0 dx$. For $n=2$, a critical mass curve was proposed:
\begin{itemize} 
   \item if $m_u m_w-2 \pi\left(m_u+m_w\right)>0$, and $\int_{\Omega} u_0\left|x-x_0\right|^2 dx$, $ \int_{\Omega} w_0\left|x-x_0\right|^2 dx$ are small enough, there exist solutions that blow up in finite time \cite{2018-N-YuWangZheng}. 
   \item if $m_u m_w-2 \pi\left(m_u+m_w\right)<0$, the solutions remain globally bounded \cite{2024-NARWA-YuXueHuZhao}.
\end{itemize}

For the system (\ref{eq1.1.0-1}) with $D_i(s) =  (s+1)^{p_i-1}$ and $S_i(s) =  s(1+s)^{q_i-1}$, Zheng \cite{2017-TMNA-Zheng} found that the system (\ref{eq1.1.0-1}) possesses a globally bounded classical solution when $q_1<p_1-1+\frac{2}{n} $ and $ q_2<p_2-1+\frac{2}{n}$. For the special case $q_i\equiv 1$, Zhong \cite{2021-JMAA-Zhong} proved that if $p_1p_2+\frac{2p_1}{n}>p_1+p_2$ or $p_1p_2+\frac{2p_2}{n}>p_1+p_2$, the solutions of (\ref{eq1.1.0-1}) exist globally and remain bounded. Recently, Zeng and Li \cite{2025--ZengLi} considered the case $p_i \equiv 1$, and obtained a critical blow-up curve (i.e. $q_1+q_2-\frac{4}{n}=\max\Big\{\big(q_1-\frac{2}{n}\big)q_2,\big(q_2-\frac{2}{n}\big)q_1\Big\}$ in the square $(0,\frac{4}{n}) \times (0,\frac{4}{n})$) to
classify two dynamics among global boundedness and finite-time blow-up of solutions.

The studies of (\ref{eq1.1.0-1}) can be traced back to the classical Keller-Segel system \cite{1971-JTB-KellerSegel,1971-Jotb-KellerSegel, 2002-CAMQ-PainterHillen}, which involves one species and one chemical
\begin{align}\label{eq1.1.1}
\begin{cases}
u_t=\nabla \cdot(D(u) \nabla u)-\nabla \cdot(S(u) \nabla v), & x \in \Omega, t>0, \\ 
0=\Delta v-v+u,  & x \in \Omega, t>0 , \\
 \frac{\partial u}{\partial \nu}=\frac{\partial v}{\partial \nu}=0, & x \in \partial \Omega, t>0, \\
u(x, 0) = u_0(x), & x \in \Omega.
\end{cases}
\end{align}
For the case $D(s) \equiv 1$, $S(s)=s$, the system (\ref{eq1.1.1}) exists critical mass phenomenon when $n=2$. Let $m:=\int_{\Omega} u_0 dx$. In the radially symmetric setting, Nagai \cite{1995-AMSA-Nagai} proved that, 
\begin{itemize}
\item when $m<8 \pi$ the solutions of system (\ref{eq1.1.1}) are globally bounded;
\item when $m>8 \pi$, there exist radially symmetric initial data with small $\int_{\Omega} u_0(x)|x|^2 d x$ such that the corresponding solution blows up in finite time. 
\end{itemize}
Later, Nagai \cite{2001-JIA-Nagai} removed the requirement for radial symmetry and established the following blow-up criterion: if $\int_{\Omega} u_0|x-q|^2 d x$ is sufficiently small, then the solution blows up in finite time under the conditions that either $q \in \Omega$ and $m>8 \pi$ or $q \in \partial \Omega$ and $m>4 \pi$. The same results are also valid for the system (\ref{eq1.1.1}) with Jäger-Luckhaus form (i.e. the second equation is replaced by $0 = \Delta v - \fint_{\Omega}u dx + u$). For the parabolic-parabolic version of (\ref{eq1.1.1}), related results can be found in \cite{1997-FE-NagaiSenbaYoshida, 2001-EJAM-HorstmannWang, 2014--MizoguchiWinkler}.

When $D(s)=(s+1)^{p}$ and $S(s)=s(s+1)^{q-1}$, the system (\ref{eq1.1.1}) exhibits two critical lines (i.e. $q-p=\frac{2}{n}$ and $q=0$). For the system (\ref{eq1.1.1}) with Jäger-Luckhaus form, Winkler and Djie \cite{2010-NA-WinklerDjie} proved that 
\begin{itemize}
\item if $q-p<\frac{2}{n}$, all solutions are globally bounded; 
\item if $q-p>\frac{2}{n}$ and $q>0$, there exist radially symmetric solutions that blow up in finite time.
\end{itemize}   
For the system (\ref{eq1.1.1}), Lankeit \cite{2020-DCDSSS-Lankeit} demonstrated that, 
\begin{itemize}
\item if $q-p>\frac{2}{n}$ and $q \leq 0$, there exist solutions that
blow up in infinite time. 
\end{itemize}
Similar results regarding the parabolic–parabolic version of (\ref{eq1.1.1}) can be found in \cite{2010-MMAS-Winkler,2012-JDE-TaoWinkler,2012-JDE-CieslakStinner,2014-AAM-CieslakStinner,2015-JDE-CieslakStinner,2019-JDE-Winkler, {2025-CVPDE-CaoFuest}}

$\mathbf{Main\ results.}$ Let $m_u:=\int_{\Omega} u_0 dx$ and $m_w=:\int_{\Omega} w_0 dx$. We assume that
\begin{align}\label{eq1.2}
\begin{aligned}
&S(s),D(s) \in C^3([0, \infty))  \text { satisfy } D(s)>0  \text { on }  [0, \infty) , \ S(s) > 0 
\ \text {on} \ (0, \infty) \text { and } S(0)=0,
\end{aligned}
\end{align}
and 
\begin{align}\label{eq1.1}
\begin{aligned}
u_0, w_0 \in  W^{1,\infty}(\overline{\Omega})\ \text { are positive }. 
\end{aligned}
\end{align}

\begin{thm}\label{thm1_1}
Let $\Omega=B_R(0) \subset \mathbb{R}^n$ $(n \geq 3)$ with some $R>0$. Suppose that $u_0$ and $w_0$ are radially symmetric, satisfying $(\ref{eq1.1})$. Assume that $D(s)$ and $S(s)$ satisfy $(\ref{eq1.2})$ as well as
\begin{align}\label{eq1.3}
D(s) \leq k_D s^p, \quad  s \geq 1
\end{align}
and
\begin{align}\label{eq1.4}
S(s) \geq k_S s^q, \quad  s \geq 1
\end{align}
with some $k_D,k_S > 0$ and $p,q\in \mathbb{R}$ fulfilling
\begin{align}\label{result-all}
q-p>2-\frac{n}{2} \quad \text{and} \quad q>1-\frac{n}{2}.
\end{align}
Then, one can find $M_1(r), M_2(r) \in C^0([0, R])$ such that, whenever radially symmetric initial data $u_0$, $w_0$ satisfying 
\begin{align}\label{eq1.8}
\int_{B_r(0)} u_0 dx\geq M_1(r),\quad \int_{B_r(0)} w_0dx \geq M_2(r),\quad  r\in(0,R),
\end{align}
the problem \eqref{eq1.1.0} admits a classical solution in $\Omega \times (0,T_{\max })$ which blows up in finite time.
\end{thm}
\begin{rem}
Our proof follows the approach to constructing subsolutions as presented in Tao and Winkler \cite{2025-JDE-TaoWinkler}.
\end{rem}
The next two theorems indicate that the ranges in (\ref{result-all}) are optimal for finite-time blow-up.
\begin{thm}\label{thm1_2}
Let $\Omega \subset \mathbb{R}^n$ $(n \geq 3)$ be a smooth bounded domain. Assume that $D(s)$ and $S(s)$ satisfy $(\ref{eq1.2})$ as well as
\begin{align}\label{D-g}
D(s) \geq K_D (1+s)^p, \quad  s \geq 0
\end{align}
and
\begin{align}\label{S-g}
S(s) \leq K_S (1+s)^{q}, \quad  s \geq 0
\end{align}
with some $K_D,K_S > 0$ and $p,q\in \mathbb{R}$ fulfilling
\begin{align*}
q-p<2-\frac{n}{2} .
\end{align*}
Then, for all initial data fulfilling $(\ref{eq1.1})$, the solution of (\ref{eq1.1.0}) exists globally and remains bounded in the sense that there exists $C>0$ independent of $t$ such that
\begin{align*}
\|u(\cdot, t)\|_{L^{\infty}(\Omega)}+\|w(\cdot, t)\|_{L^{\infty}(\Omega)} \leq C,
\quad t>0.
\end{align*}
\end{thm}
By suitably modifying the argument underlying Theorem~\ref{thm1_2}, we show that the second inequality in (\ref{result-all}) cannot be weakened. 
\begin{thm}\label{thm1_3}
Let $\Omega \subset \mathbb{R}^n$ $(n \geq 3)$ be a smooth bounded domain.  Assume that $D(s)$ and $S(s)$ satisfy $(\ref{eq1.2})$ as well as $(\ref{D-g})$ and $(\ref{S-g})$
with some $K_D,K_S > 0$ and $p,q\in \mathbb{R}$ fulfilling
\begin{align*}
q<1-\frac{n}{2} .
\end{align*}
Then, for all initial data fulfilling $(\ref{eq1.1})$, the corresponding classical solution of system $(\ref{eq1.1.0})$ exists globally.
\end{thm}
Our results can be summarized in the Figure~\ref{fig:mq_graph}. 


\begin{figure}[!htbp]
\begin{center}
\begin{minipage}{0.4\textwidth}
\begin{center}
\begin{tikzpicture}[scale=1.3]
\draw[->, thick] (-2.5, 0) -- (1, 0) node[right]{$p$};
\draw[->, thick] (0, -2.5) node[below=1em]{ } -- (0, 1) node[above]{$q$};
\draw[very thick] (-2, -2.5) -- (1, 0.5);
\draw[very thick] (-2.5, -1.5) -- (-1, -1.5);
\draw[thick] (0.5, 0) -- (0.5, 0.05);
\draw[thick] (0, -0.5) -- (0.05, -0.5);
\node[] at (-2.1, -0.75) {\textbf{FTBU}};
\node[] at (0.75, -0.26) {$\frac{n}{2}-2$};
\node[] at (-0.5, -0.5) {$2-\frac{n}{2}$};
\node[] at (-0.2, -1.5) {$q=1-\frac{n}{2}$};
\node[] at (-2.2, -2) {\textbf{GE}};
\node[] at (2, 0.7) {$q-p=2-\frac{n}{2}$};
\node[] at (0.5, -1.0) {\textbf{GB}};
\end{tikzpicture}
\end{center}
\end{minipage}
\end{center}
\captionsetup{type=figure}
\captionof{figure}{
“GB”: All solutions are globally bounded. “GE”: All solutions exist globally. “FTBU”: There exist solutions that blow up in finite time. \\}
\label{fig:mq_graph}
\end{figure}
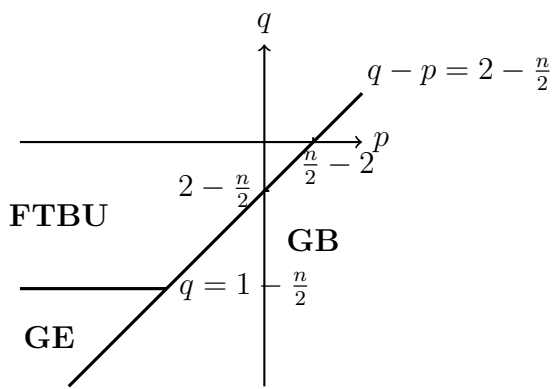



The rest of the paper is organized as follows. We present the local existence and the extensibility criterion of the classical solution in Section~\ref{local existence}. The purpose of Section~\ref{blowup} is to prove that the system (\ref{eq1.1.0}) admits finite-time blow-up solutions. We first establish a weak comparison principle in Subsection~\ref{comparison}, then construct subsolutions, which become singular within finite time, to prove Theorem~\ref{thm1_1} in Subsection~\ref{subsolution}. Section~\ref{bounded} is devoted to showing the global boundedness of solutions. Section~\ref{global existence} is concerned with deriving the global existence of solutions.

\section{local existence and extensibility}\label{local existence}
We first recall the basic result on the local existence, uniqueness and
extensibility criterion of the classical solution of system (\ref{eq1.1.0}). The proof of this result follows the arguments in \cite{2015-DCDSSB-TaoWinkler,2010-NA-WinklerDjie,2025--ZengLi}. 
\begin{prop}\label{local}
Let $\Omega \subset \mathbb{R}^n$ $(n \geq 1)$ be a smooth bounded domain. Suppose that $(\ref{eq1.2})$ and $(\ref{eq1.1})$ are valid. Then one can find $T_{\max } \in(0, \infty]$ and uniquely determined functions 
\begin{align*}
\begin{aligned}
& u \in C^0\left(\overline{\Omega} \times\left[0, T_{\max }\right)\right) \cap C^{2,1}\left(\overline{\Omega} \times\left(0, T_{\max }\right)\right), \\
& v \in C^{2,0}\left(\overline{\Omega} \times\left[0, T_{\max }\right)\right), \\
& w \in C^0\left(\overline{\Omega} \times\left[0, T_{\max }\right)\right) \cap C^{2,1}\left(\overline{\Omega} \times\left(0, T_{\max }\right)\right), \\
& z \in C^{2,0}\left(\overline{\Omega} \times\left[0, T_{\max }\right)\right),
\end{aligned}
\end{align*}
solving \eqref{eq1.1.0} in the classical sense in $\Omega \times\left(0, T_{\max }\right)$ and fulfilling the following extensibility property:
\begin{align}\label{extensibility}
\text { if } T_{\max }<\infty \text {, then } \limsup _{t \nearrow T_{\max }}\left(\|u(\cdot, t)\|_{L^{\infty}(\Omega)}+\|w(\cdot, t)\|_{L^{\infty}(\Omega)}\right)=\infty \text {. }
\end{align}
Moreover, $u, w>0$ in $\overline{\Omega} \times [0,T_{\max})$ and $\int_{\Omega} v(\cdot,t) dx=\int_{\Omega} z(\cdot,t) dx=0$ for all $t \in (0,T_{\max})$, as well as
\begin{align*}
\int_\Omega u(t) dx=\int_\Omega u_0 dx \ \ \text{and} \ \
\int_\Omega w(t) dx=\int_\Omega w_0  dx, \quad t \in (0,T_{\max}).
\end{align*}

If in addition $\Omega=B_R(0)$ for some $R>0$, and $\left(u_0, w_0\right)$ is a triplet of radially symmetric functions, then $u, v, w, z$ are all radially symmetric.
\end{prop}

\section{Finite-time blow-up. Proof of Theorem \ref{thm1_1} }\label{blowup}
\subsection{Comparison principle}\label{comparison}
Throughout this section, the domain and the non-trivial initial functions $u_0$ and $w_0$ are radially symmetric. 

Let 
\begin{align}\label{eq2_6}
\mu^{\star}:=\max\left\{\fint_{\Omega} u_0 ,\fint_{\Omega} w_0 \right\},\quad
\mu_{\star}:=\min\left\{\fint_{\Omega} u_0 ,\fint_{\Omega} w_0 \right\}.
\end{align}
For positive functions $\varphi, \psi\in C^0\left([0, T);C^1\left[0, R^n\right] \right)$, which satisfy $\varphi_s \geq 0$ and $\psi_s \geq 0$ on $\left(0, R^n\right) \times(0, T)$, and such that $\varphi(\cdot, t), \psi(\cdot, t) \in W_{l o c}^{2, \infty}\left(\left(0, R^n\right)\right)$ for all $t \in(0, T)$, we define the differential operators $\mathcal{P}$ and $\mathcal{Q}$ 
\begin{align}\label{eq2.3}
\begin{cases}
\begin{aligned}
&\mathcal{P}[\varphi, \psi](s, t):=\varphi_t-n^2 s^{2-\frac{2}{n}}D(n\varphi_{s}) \varphi_{s s}-S\left(n \varphi_s\right) \cdot\left(\psi-\frac{\mu^{\star} s}{n}\right), \\
& \mathcal{Q}[\varphi, \psi](s, t):=\psi_t-n^2 s^{2-\frac{2}{n}} \psi_{s s}-n \psi_s \cdot\left(\varphi-\frac{\mu^{\star} s}{n}\right),
\end{aligned}
\end{cases}
\end{align}
for $t \in(0, T)$ and a.e. $s \in\left(0, R^n\right)$.

The following transformation of (\ref{eq1.1.0}) adapts a meanwhile quite classical observation \cite{1992-TAMS-JaegerLuckhaus} to the present framework.
\begin{lem}
Let $\Omega = B_{R}(0) \subset \mathbb{R}^n$ $(n \geq 1)$ and $T>0$. Suppose that $(\ref{eq1.2})$ and $(\ref{eq1.1})$ are valid. Let
\begin{align}\label{eq2.1}
U(s,t):=\int^{s^\frac{1}{n}}_0 r^{n-1}u(r,t)dr,\quad
W(s,t):=\int^{s^\frac{1}{n}}_0 r^{n-1}w(r,t)dr
\end{align}
for all $s \in\left[0, R^n\right]$ and $ t \in\left[0, T_{\max }\right)$. Then, we have
\begin{align}\label{eq2_5}
\begin{split}
\begin{cases}
\mathcal{P}[U, W](s, t) \geq 0,\quad &s\in(0,R^n), t\in(0,T_{\max}), \\
\mathcal{Q}[U, W](s, t) \geq 0, \quad &s\in(0,R^n), t\in(0,T_{\max}),\\
U(0,t)=W(0,t)=0,\quad &t\in(0,T_{\max}),\\
U(R^n,t)=\frac{\mu_uR^n}{n}\geq\frac{\mu_{\star}R^n}{n},\ W(R^n,t)=\frac{\mu_wR^n}{n}\geq\frac{\mu_{\star}R^n}{n},\quad &t\in(0,T_{\max}),\\
U(s,0)=\int^{s^\frac{1}{n}}_0 r^{n-1}u_0(r,t)dr,\quad &s\in(0,R^n),\\
W(s,0)=\int^{s^\frac{1}{n}}_0 r^{n-1}w_0(r,t)dr,\quad &s\in(0,R^n).
\end{cases}
\end{split}
\end{align}
\end{lem}
\begin{proof}
Using (\ref{eq1.1.0}) and (\ref{eq2.1}), (\ref{eq2_5}) can be easily verified (cf. also \cite{2010-NA-WinklerDjie}).
\end{proof}
The parabolic system associated with (\ref{eq2_5}) permits a comparison argument, which plays an important role in the subsequent analysis of the explosion.
\begin{lem}\label{lem2.2}
Let $T>0$ and $\Omega = B_{R}(0) \subset \mathbb{R}^n$ $(n\geq 1)$. Assume that $D(s)$ and $S(s)$ satisfy $(\ref{eq1.2})$. Let $\underline{U}, \overline{U}, \underline{W}, \overline{W} \in C^1\left(\left[0, R^n\right] \times [0, T) \right)$ such that $\underline{U}_s, \overline{U}_s, \underline{W}_s, \overline{W}_s\geq0$ for $(s,t)\in(0,R^n) \times (0,T)$ as well as $\underline{U}(\cdot,t), \overline{U}(\cdot,t), \underline{W}(\cdot,t), \overline{W}(\cdot,t) \in W_{l o c}^{2, \infty}\left(\left(0, R^n\right)\right)$ for $t \in (0,T)$. Assume that
\begin{align}\label{eq2_9}
\begin{array}{lll}
\underline{U}(0, t) \leq \overline{U}(0, t), & \underline{U}\left(R^n, t\right) \leq \overline{U}\left(R^n, t\right), \quad t \in[0, T), \\
\underline{W}(0, t) \leq \overline{W}(0, t), & \underline{W}\left(R^n, t\right) \leq \overline{W}\left(R^n, t\right), \quad t \in[0, T),
\end{array}
\end{align}
and 
\begin{align}\label{eq2 9-1}
 \underline{U}(s, 0) \leq \overline{U}(s, 0) \ \text{and} \ \  \underline{W}(s, 0) \leq \overline{W}(s, 0),  \quad  s \in [0,R^n].
\end{align}
as well as 
\begin{align}\label{eq2_8}
\begin{array}{ll}
\mathcal{P}[\underline{U}, \underline{W}](s, t) \leq 0, & \mathcal{P}[\overline{U}, \overline{W}](s, t) \geq 0, \\
\mathcal{Q}[\underline{U}, \underline{W}](s, t) \leq 0, & \mathcal{Q}[\overline{U}, \overline{W}](s, t) \geq 0,
\end{array}
\end{align}
for all $t \in(0, T)$ and a.e. $s \in\left(0, R^n\right)$, then
\begin{align}\label{eq2_10}
\underline{U}(s, t) \leq \overline{U}(s, t), \quad
\underline{W}(s, t) \leq \overline{W}(s, t), 
\quad  (s,t) \in\left[0, R^n\right] \times [0, T).
\end{align}
\end{lem}
\begin{proof}
For any $T_0 \in (0, T)$, the functions $X(s,t)$ and $Y(s,t)$ defined by 
\begin{align*}
	&X(s, t) := \underline{U}(s, t) - \overline{U}(s, t) - \varepsilon \mathrm{e}^{\lambda t}, \quad (s,t) \in [0, R^n] \times [0, T_0], \\
   &Y(s, t) := \underline{W}(s, t) - \overline{W}(s, t) - \varepsilon \mathrm{e}^{\lambda t}, \quad (s,t) \in [0, R^n] \times [0, T_0]
\end{align*}
with some $\varepsilon,\lambda> 0$. Due to \eqref{eq2_9} and \eqref{eq2 9-1}, we infer that $X(0,t), Y(0,t)<0$ and $X(R^n,t)$, $Y(R^n,t)<0$ for all $t \in [0, T_0]$ as well as $X(s,0), Y(s,0)<0$ for all $s \in [0, R^n]$. We claim that
\begin{align}\label{eq2_12}
\max _{(s, t) \in[0, R^n] \times[0, T_0]}\{ X(s,t),Y(s,t)\} < 0, 
\end{align}
when 
\begin{align}\label{eq2_11}
\lambda \geq \max\{2\|S\left(n\underline{U}_s\right)\|_{L^{\infty}([0, R^n] \times [0, T_0])}, 2\|n\underline{W}_s\|_{L^{\infty}([0, R^n] \times [0, T_0])}\}.
\end{align}
Indeed, if (\ref{eq2_12}) were false, at least one of the following two cases must hold.

$\textbf{Case\ 1.}$ There exist $s_X\in(0,R^n)$ and $t_0\in(0,T_0]$ such that
\begin{align}\label{eq2_13}
\max _{(s, t) \in[0, R^n] \times[0, t_0]}\{ X(s,t),Y(s,t)\}=X(s_X,t_0)=0.
\end{align}

Then, we have
	\begin{align}\label{eq2_15}
		X_t(s_X, t_0) \geq 0
	\end{align}
	and
	\begin{align}\label{eq2_16}
		X_s(s_X, t_0) = 0.
	\end{align}
Moreover, since $X(\cdot, t_0) \in W_{\text{loc}}^{2, \infty}\left((0, R^n)\right)$, we can find a null set $N(t_0) \subset (0, R^n)$ such that $X_{ss}(s, t_0)$ exists for $s \in (0, R^n) \backslash N(t_0)$. Owing to $\left(0, R^n\right) \backslash N(t_0)$ is dense in $\left(0, R^n\right)$, along with \eqref{eq2_13}, we can fix $\left(s_j\right)_{j \in \mathbb{N}} \subset (0, R^n) \backslash N(t_0)$ such that $s_j \rightarrow s_X$ as $j \rightarrow \infty$ and
\begin{align}\label{eq2_18}
\liminf _{j \rightarrow \infty} X_{ss}(s_j, t_0) \leq 0.
\end{align}
According to \eqref{eq2_8}, we obtain
\begin{align}\label{Xt}
		X_t(s_j, t_0) &= \underline{U}_t(s_j, t_0) - \overline{U}_t(s_j, t_0) - \lambda \varepsilon \mathrm{e}^{\lambda t_0} \nonumber \\
& \leq n^2 {s_j}^{2-\frac{2}{n}}\left(D\left(n \underline{U}_s\left(s_j, t_0\right)\right) \underline{U}_{s s}\left(s_j, t_0\right)-D\left(n \overline{U}_s\left(s_j, t_0\right)\right) \overline{U}_{s s}\left(s_j, t_0\right)\right) \nonumber\\
& \quad +S\left(n \underline{U}_s\left(s_j, t_0\right)\right)\left(\underline{W}\left(s_j, t_0\right)-\frac{\mu^{\star} s_j}{n}\right)-S\left(n \overline{U}_s\left(s_j, t_0\right)\right)\left(\overline{W}\left(s_j, t_0\right)-\frac{\mu^{\star} s_j}{n}\right) \nonumber\\
& \quad -\lambda \varepsilon e^{\lambda t_0}\nonumber \\
& =  n^2 {s_j}^{2-\frac{2}{n}} D\left(n \underline{U}_s\left(s_j, t_0\right)\right) X_{s s}\left(s_j, t_0\right) \nonumber \\
& \quad +n^2 {s_j}^{2-\frac{2}{n}} \left(D\left(n \underline{U}_s\left(s_j, t_0\right)\right)-D\left(n \overline{U}_s\left(s_j, t_0\right)\right)\right) \cdot \overline{U}_{s s}\left(s_j, t_0\right) \nonumber \\
&\quad+  S\left(n \underline{U}_s\left(s_j, t_0\right)\right) \left(\underline{W}\left(s_j, t_0\right)-\frac{\mu^* s_j}{n}\right)-S\left(n \overline{U}_s\left(s_j, t_0\right)\right) \left(\overline{W}\left(s_j, t_0\right)-\frac{\mu^* s_j}{n}\right)\nonumber \\
& \quad-\lambda \varepsilon e^{\lambda t_0}
.
\end{align}
Due to $\overline{U}_s(s_X, t_0)= \underline{U}_s(s_X, t_0) \geq 0$, as asserted by (\ref{eq2_16}), and the continuity of $D(s),S(s)$ on $(0,\infty)$, we infer that
\begin{align*}
D\left(n \underline{U}_s\left(s_j, t_0\right)\right)-D\left(n \overline{U}_s\left(s_j, t_0\right)\right) \rightarrow 0 
\ \text{and} \ S\left(n \underline{U}_s\left(s_j, t_0\right)\right)-S\left(n \overline{U}_s\left(s_j, t_0\right)\right) \rightarrow 0,
\ \text{as} \ j \rightarrow \infty.
\end{align*}
Thus, using (\ref{eq2_18}) and taking $j \rightarrow \infty$ in (\ref{Xt}), we obtain
	\begin{align*}
		X_t(s_X, t_0) 
	&\leq S\left(n\underline{U}_s(s_X,t_0)\right)
	\left(\underline{W}(s_X,t_0)-\overline{W}(s_X,t_0)\right)-\lambda\varepsilon e^{\lambda t_0}\\
    &=S\left(n\underline{U}_s(s_X,t_0)\right) (Y(s_X, t_0) + \varepsilon \mathrm{e}^{\lambda t_0}) -\lambda\varepsilon e^{\lambda t_0}.
	\end{align*}
	Since $ Y(s_X, t_0) \leq 0$ by \eqref{eq2_13}, along with \eqref{eq2_11} and \eqref{eq2_15}, we have
	\begin{align*}
		0 \leq X_t(s_X, t_0) &\leq S\left(n \underline{U}_s(s_X, t_0)\right) \varepsilon \mathrm{e}^{\lambda t_0} - \lambda \varepsilon \mathrm{e}^{\lambda t_0} \leq -\frac{\lambda \varepsilon \mathrm{e}^{\lambda t_0}}{2},
	\end{align*}
which implies a contradiction. Thus, (\ref{eq2_13}) is false.
 
$\textbf{Case\ 2.}$ There exist $s_Y\in(0,R^n)$ and $t_1\in(0,T_0]$ such that
\begin{align}\label{eq2_14}
\max _{(s, t) \in[0, R^n] \times[0, t_1]}\{ X(s,t),Y(s,t)\}=Y(s_Y,t_1)=0.
\end{align}
As in Case 1, along with $\lambda \geq 2\|n\underline{W}_s\|_{L^{\infty}([0, R^n] \times [0, T_0])}$, we can conclude that (\ref{eq2_14}) is also false.

Combining these two cases, we therefore obtain \eqref{eq2_12}. As a consequence of letting $\varepsilon \searrow 0$ and $T_0 \nearrow T$, \eqref{eq2_10} follows.
\end{proof}

\subsection{ Construction of subsolution pairs}\label{subsolution}
In this section, based on the comparison principle, we construct subsolutions inspired by \cite{2025-JDE-TaoWinkler} to detect finite-time blow-up.



The following lemma selects three parameters $\delta$, $\alpha$ and $\beta$, which will be used in the construction of subsolutions. The choices
of parameters $\delta$, $\alpha$ and $\beta$ here are different from those in \cite{2025-JDE-TaoWinkler}.
\begin{lem}\label{alphabeta}
Let $n \geq 3$. Suppose that $p,q$ satisfy 
\begin{align}\label{result}
q-p>2-\frac{n}{2} \quad \text{and} \quad q>1-\frac{n}{2},
\end{align}
then there exist $\beta \in (0,1)$ and 
\begin{align}\label{alpha}
\alpha \in \left(0,1-\frac{2}{n}\right)\ \text{ and } \ \delta \in (0,\frac{2}{n})
\end{align}
such that
\begin{align}\label{q1}
(1-\alpha)q+\alpha-\beta-\delta>0,
\end{align}
and
\begin{align}\label{q3}
(1-\alpha)(q-p)+\alpha-\beta-\frac{2}{n}>0.
\end{align}
\end{lem}
\begin{proof}
According to (\ref{result}), we have $(1-\alpha)q+\alpha-\beta-\delta \rightarrow \frac{2}{n}(q-(1-\frac{n}{2}))>0$, $(1-\alpha)(q-p)+\alpha-\beta-\frac{2}{n} \rightarrow \frac{2}{n}((q-p)-(2-\frac{n}{2})) >0$ as $(\delta, \alpha, \beta) \rightarrow (0, 1-\frac{2}{n}, 0)$. Then, there exist positive constants $\delta_1 \in (0,\frac{2}{n})$, $\alpha_1 \in (0,1-\frac{2}{n})$ and $\beta_1\in (0,\frac{1}{2})$ such that (\ref{q1}) and (\ref{q3}) hold for $\delta \in (0,\delta_1)$ , $\alpha \in (\alpha_1, 1-\frac{2}{n})$ and $\beta \in (0,\beta_1)$.
\end{proof}
We are now ready to specify the subsolutions for (\ref{eq2_5}). Let $\theta>1$ be a positive constant, which will be determined later in Lemma~\ref{lem3.4} and let $\alpha, \beta \in (0,1)$ be from Lemma~\ref{alphabeta}. For any $T>0$ and $y \in C^1([0, T))$ with $y(t)>\frac{1}{R^n}$ for all $t \in (0,T)$, we define $$\underline{U}, \underline{W} \in C^1([0, R^n] \times [0, T)) \cap C^0([0, T); W^{2, \infty}((0, R^n))) \cap C^2\Big( [0, R^n] \setminus \Big\{\frac{1}{y(t)}\Big\} \Big)$$ as follows
\begin{equation}\label{eq3_7}
	\begin{cases}
		\begin{array}{ll}
			\underline{U}(s, t) := \mathrm{e}^{-\theta t} \Phi(s, t), & s \in \left[0, R^n\right], t \in [0, T), \\
			\underline{W}(s, t) := \mathrm{e}^{-\theta t} \Psi(s, t), & s \in \left[0, R^n\right], t \in [0, T)
		\end{array}
	\end{cases}
\end{equation}
with 
\begin{align}\label{phi}
	& \Phi(s, t) =
	\begin{cases}
		l y^{1 - \alpha}(t) s, & t \in [0, T), s \in \left[0, \frac{1}{y(t)}\right], \\
		l \alpha^{-\alpha} \cdot \left(s - \frac{1 - \alpha}{y(t)}\right)^\alpha, & t \in [0, T), s \in \left(\frac{1}{y(t)}, R^n\right],
	\end{cases} 
\end{align}
\begin{align}\label{psi}
	& \Psi(s, t) =
	\begin{cases}
		l y^{1 - \beta}(t) s, & t \in [0, T), s \in \left[0, \frac{1}{y(t)}\right], \\
		l \beta^{-\beta} \cdot \left(s - \frac{1 - \beta}{y(t)}\right)^\beta, & t \in [0, T), s \in \left(\frac{1}{y(t)}, R^n\right],
	\end{cases}
\end{align}
where
\begin{align}\label{eq3_4}
	l = \frac{\mu_{\star} R^n}{n\mathrm{e}^{\frac{1}{\mathrm{e}}}(R^n + 1)}
\end{align}
with $\mu_{\star}$ defined in (\ref{eq2_6}). Then, we have
\begin{align}\label{phis}
	& \Phi_s(s, t) =
	\begin{cases}
		l y^{1 - \alpha}(t), & t \in (0, T), s \in \left(0, \frac{1}{y(t)}\right), \\
		l \alpha^{1 - \alpha} \cdot \left(s - \frac{1 - \alpha}{y(t)}\right)^{\alpha - 1}, & t \in (0, T), s \in \left(\frac{1}{y(t)}, R^n\right),
	\end{cases} 
\end{align}
\begin{align}\label{psis}
	& \Psi_s(s, t) =
	\begin{cases}
		l y^{1 - \beta}(t), & t \in (0, T), s \in \left(0, \frac{1}{y(t)}\right), \\
		l \beta^{1 - \beta} \cdot \left(s - \frac{1 - \beta}{y(t)}\right)^{\beta - 1}, & t \in (0, T), s \in \left(\frac{1}{y(t)}, R^n\right)
	\end{cases}
\end{align}
and
\begin{align}\label{phiss}
	& \Phi_{ss}(s, t) =
	\begin{cases}
		0, & t \in (0, T), s \in \left(0, \frac{1}{y(t)}\right), \\
		l \alpha^{1 - \alpha} (\alpha - 1) \cdot \left(s - \frac{1 - \alpha}{y(t)}\right)^{\alpha - 2}, & t \in (0, T), s \in \left(\frac{1}{y(t)}, R^n\right),
	\end{cases} 
\end{align}
\begin{align}\label{psiss}
	& \Psi_{ss}(s, t) =
	\begin{cases}
		0, & t \in (0, T), s \in \left(0, \frac{1}{y(t)}\right), \\
		l \beta^{1 - \beta} ( \beta-1) \cdot \left(s - \frac{1 - \beta}{y(t)}\right)^{\beta - 2}, & t \in (0, T), s \in \left(\frac{1}{y(t)}, R^n\right),
	\end{cases}
\end{align}
as well as
\begin{align}\label{phit}
	& \Phi_t(s, t) =
	\begin{cases}
		l (1 - \alpha) y^{\prime}(t) y^{-\alpha}(t) s, & t \in (0, T), s \in \left(0, \frac{1}{y(t)}\right), \\
		l \alpha^{1 - \alpha} (1 - \alpha) \cdot \left(s - \frac{1 - \alpha}{y(t)}\right)^{\alpha - 1} \frac{y^{\prime}(t)}{y^2(t)}, & t \in (0, T), s \in \left(\frac{1}{y(t)}, R^n\right),
	\end{cases} 
\end{align}
\begin{align}\label{psit}
	& \Psi_t(s, t) =
	\begin{cases}
		l (1 - \beta) y^{\prime}(t) y^{-\beta}(t) s, & t \in (0, T), s \in \left(0, \frac{1}{y(t)}\right), \\
		l \beta^{1 - \beta} (1 - \beta) \cdot \left(s - \frac{1 - \beta}{y(t)}\right)^{\beta - 1} \frac{y^{\prime}(t)}{y^2(t)}, & t \in (0, T), s \in \left(\frac{1}{y(t)}, R^n\right).
	\end{cases}
\end{align}
\begin{lem}\label{s}
Let $n \geq 3$ and $\alpha, \beta$ be taken from Lemma~\ref{alphabeta}. Then, $\underline{U}$ and $\underline{W}$ satisfy
\begin{align}\label{s=0}
\underline{U}(0, t) = 0, 
\quad \underline{W}(0, t) = 0,
\quad t \in [0, T)
\end{align}
and
\begin{align*}
\underline{U}(R^n, t) \leq \frac{\mu_{\star} R^n}{n}, \quad \underline{W}(R^n, t) \leq \frac{\mu_{\star} R^n}{n},
\quad t \in [0, T),
\end{align*}
as well as
\begin{align*}
\underline{U}(s, 0) \leq U(s, 0), \quad 
\underline{W}(s, 0) \leq W(s, 0),
\quad s \in [0,R^n].
\end{align*}
\end{lem}
\begin{proof}
By the definition of $\underline{U}$ and $\underline{W}$, we deduce (\ref{s=0}). Applying the definition of $\underline{U}$ and $l$ in (\ref{eq3_4}), along with $\alpha^{-\alpha}=\mathrm{e}^{-\alpha\mathrm{ln}\alpha} \leq \mathrm{e}^{\frac{1}{\mathrm{e}}}$, we deduce that
	\begin{align*}
		\begin{aligned}
			\underline{U}(R^n, t) &= \mathrm{e}^{-\theta t} \alpha^{-\alpha} l \left( R^n - \frac{1 - \alpha}{y(t)} \right)^{\alpha} 
			\leq \alpha^{-\alpha} l R^{n\alpha} 
			= \alpha^{-\alpha} R^{n\alpha} \frac{\mu_{\star} R^n}{n \mathrm{e}^{\frac{1}{\mathrm{e}}}(R^n + 1)} \\
			&\leq \frac{\mu_{\star} R^n}{n} \cdot \frac{R^{n\alpha}}{R^n + 1} 
\leq \frac{\mu_{\star} R^n}{n}.
		\end{aligned}
	\end{align*}
In \eqref{eq1.8}, we define $M_1(r)$ and $M_2(r)$ as follows
\begin{align*}
M_1(r) = \omega_n  \underline{U}(r^n, 0), \quad \quad M_2(r) = \omega_n \underline{W}(r^n, 0),
 \quad  r \in [0, R],
\end{align*}
where $\omega_n$ is the surface area of the unit sphere. Thus, we have
\begin{align*}
U(s, 0)  =\frac{1}{\omega_n} \int_{B_{s^{\frac{1}{n}}}(0)} u_0 \mathrm{d} x 
\geq \frac{1}{\omega_n} \cdot M_1\left(s^{\frac{1}{n}}\right)
=\underline{U}(s, 0),
\quad  s \in\left[0, R^n\right].
\end{align*}
Similarly, we obtain $\underline{W}(R^n, t) \leq \frac{\mu_{\star} R^n}{n}$ for all $t\in[0,T)$ and $W(s, 0) \geq \underline{W}(s, 0)$ for all $s \in\left[0, R^n\right]$. We complete our proof.
\end{proof}

We proceed by dividing $[0, R^n]$ into three regions. In each of these regions, we will prove $\mathcal{P}[\underline{U}, \underline{W}](s, t) \leq 0$ and $\mathcal{Q}[\underline{U}, \underline{W}](s, t) \leq 0$. We begin by verifying these near $s = 0$.

\begin{lem}\label{lem3.1}
Let $\Omega = B_{R}(0) \subset \mathbb{R}^n$ $(n \geq 3)$. Assume that $D(s)$ and $S(s)$ satisfy $(\ref{eq1.2})$ as well as \eqref{eq1.3} and \eqref{eq1.4} with some $k_D,k_S > 0$ and $p,q\in \mathbb{R}$ fulfilling $(\ref{result-all})$. Suppose that $\delta$, $\alpha$ and $\beta$ are taken from Lemma~\ref{alphabeta}. There exists $y_0=y_0(\alpha,\beta,l,\mu^{\star})$ satisfying
\begin{align}\label{ystar}
		y_0 > y_{\star}:= \max\Big\{1,\frac{1}{R^n},\left(\frac{\mathrm{e}}{nl}\right)^{\frac{1}{1-\alpha}},\left(\frac{2\mu^{\star}\mathrm{e}}{nl}\right)^{\frac{1}{1-\beta}},\left(\frac{2\mu^{\star}\mathrm{e}}{nl}\right)^{\frac{1}{1-\alpha}}\Big\},
\end{align}
such that if $T>0$ and $y(t) \in C^1([0, T))$ satisfy
\begin{align}\label{y1}
\left\{\begin{array}{l}
0 \leq y^{\prime}(t)\leq 
\min\big\{\frac{k_Sn^ql^{q}}{2\mathrm{e}^{|q|+1}}  , \frac{ nl}{2\mathrm{e}^{2}}  \big\}y^{1+\delta}(t), \quad t\in (0,T), \\
y(0) \geq y_0,
\end{array}\right.
\end{align}
then, for any $\theta>0$, the functions $\underline{U}$ and $\underline{W}$, defined in $(\ref{eq3_7})$, have the following properties
\begin{align*}
\mathcal{P}[\underline{U}, \underline{W}](s, t) \leq 0, \quad \mathcal{Q}[\underline{U}, \underline{W}](s, t) \leq 0, 
\end{align*}
for all $t \in (0,T) \cap \left(0, \frac{1}{\theta}\right)$ and $s \in \left(0, \frac{1}{y(t)}\right)$.
\end{lem}
\begin{proof} 
According to the second restriction in (\ref{ystar}), we infer that $y(t) > \frac{1}{R^n}$, which implies $(0,\frac{1}{y(t)}) \subset (0,R^n)$.
Recalling (\ref{eq2.3}) and (\ref{phis})-(\ref{psit}), we see that
\begin{align}\label{eq3_10}
\mathcal{P}[\underline{U}, \underline{W}](s, t) &= \underline{U}_t - n^2 s^{2 - \frac{2}{n}} D(n\underline{U}_{s}) \underline{U}_{ss} - S\left(n \underline{U}_s\right) \cdot \Big(\underline{W} - \frac{\mu^{\star} s}{n}\Big)
\nonumber \\
&= -\theta \mathrm{e}^{-\theta t} l y^{1 - \alpha}(t) s +
\mathrm{e}^{-\theta t} l (1 - \alpha) y^{-\alpha}(t) y^{\prime}(t) s \nonumber \\
&\quad - S\left(n \mathrm{e}^{-\theta t} l y^{1 - \alpha}(t)\right) \Big(\mathrm{e}^{-\theta t} l y^{1 - \beta}(t) s - \frac{\mu^{\star} s}{n}\Big) \nonumber \\
&\leq \mathrm{e}^{-\theta t} l (1 - \alpha) y^{-\alpha}(t) y^{\prime}(t) s
- S\left(n \mathrm{e}^{-\theta t} l y^{1 - \alpha}(t)\right) \Big(\mathrm{e}^{-\theta t} l y^{1 - \beta}(t) s - \frac{\mu^{\star} s}{n}\Big)
\end{align}
for $t \in (0, T)$ and $s \in \left(0, \frac{1}{y(t)}\right)$. Due to $t\in (0,T) \cap (0,\frac{1}{\theta})$, we have
\begin{align}\label{thetat}
\theta  t<1,\quad t\in (0,T).
\end{align}
Combining this with the third restriction in \eqref{ystar}, we derive that
\begin{align*}
n \mathrm{e}^{-\theta t} l y^{1 - \alpha}(t) 
>n \mathrm{e}^{-1} l y_{0}^{1 - \alpha} 
>n \mathrm{e}^{-1} l y_{\star}^{1 - \alpha} 
>n \mathrm{e}^{-1} l \cdot \Big(\frac{\mathrm{e}}{nl}\Big) = 1.
\end{align*}
According to (\ref{thetat}) and the fourth restriction in \eqref{ystar}, we deduce that
\begin{align*}
\frac{\mathrm{e}^{-\theta t} l y^{1 - \beta}(t) s}{2}
> \frac{\mathrm{e}^{-1} l y_{0}^{1 - \beta} s}{2}  
> \frac{\mathrm{e}^{-1} l y_{\star}^{1 - \beta} s}{2} 
>\frac{\mu^{\star} s}{n}.
\end{align*}
We estimate the last term on the right-hand side of (\ref{eq3_10}) by using the above two estimates, along with \eqref{eq1.4} and (\ref{thetat}), we infer that
\begin{align}\label{eq3_11}
S\big(n \mathrm{e}^{-\theta t} l y^{1 - \alpha}(t)\big) \Big(\mathrm{e}^{-\theta t} l y^{1 - \beta}(t) s - \frac{\mu^{\star} s}{n}\Big) 
&\geq k_S \big(n \mathrm{e}^{-\theta t} l y^{1 - \alpha}(t)\big)^q \cdot \frac{\mathrm{e}^{-\theta t} l y^{1 - \beta}(t) s}{2} \nonumber \\
&\geq \frac{k_Sn^ql^{q+1}}{2\mathrm{e}^{|q|+1}} \cdot  y^{(1 - \alpha)q  - \beta+ 1}(t) s. 
\end{align}
Substituting \eqref{eq3_11} into \eqref{eq3_10}, along with $y(t) \geq y_0>1$ and (\ref{q1}), we obtain
\begin{align}
\mathcal{P}[\underline{U}, \underline{W}](s, t)
&\leq l y^{-\alpha}(t) y^{\prime}(t) s 
- \frac{k_Sn^ql^{q+1}}{2\mathrm{e}^{|q|+1}} \cdot  y^{(1 - \alpha)q  - \beta+ 1}(t) s \nonumber \\
&=  l y^{-\alpha}(t) s 
\Big(y^{\prime}(t)- \frac{k_Sn^ql^{q}}{2\mathrm{e}^{|q|+1}} y^{(1 - \alpha)q+\alpha - \beta+1}(t)\Big) \nonumber \\
&\leq l y^{-\alpha}(t) s 
\Big(y^{\prime}(t)- \frac{k_Sn^ql^{q}}{2\mathrm{e}^{|q|+1}} y^{1+\delta}(t)\Big),
\end{align}
which implies
\begin{align*}
\mathcal{P}[\underline{U}, \underline{W}](s, t) \leq 0,
\quad \text{ for all } \ t \in (0,T) \cap \left(0, \frac{1}{\theta}\right) \
\text{ and }\ s \in \left(0, \frac{1}{y(t)}\right).
\end{align*}
Similarly, applying the last restriction in (\ref{ystar}) and $2-\alpha>1+\frac{2}{n}>1+\delta$ due to (\ref{alpha}), we have
\begin{align*}
\mathcal{Q}[\underline{U}, \underline{W}](s, t)
=& \underline{W}_t - n^2 s^{2 - \frac{2}{n}} \underline{W}_{ss} - n \underline{W}_s \cdot \big(\underline{U} - \frac{\mu^{\star} s}{n}\big) \\
\leq & l y^{-\beta}(t) s 
\Big(y^{\prime}(t)- \frac{nl}{2\mathrm{e}^{2}} y^{2-\alpha}(t)\Big) \\
\leq & l y^{-\beta}(t) s 
\Big(y^{\prime}(t)- \frac{nl}{2\mathrm{e}^{2}} y^{1+\delta}(t)\Big),
\end{align*}
for all $t \in (0,T) \cap \left(0, \frac{1}{\theta}\right)$ and $s \in \left(0, \frac{1}{y(t)}\right)$. We complete our proof.
\end{proof}
For our choice of $\delta, \alpha, \beta$, we also have $\mathcal{P}[\underline{U}, \underline{W}](s, t) \leq 0$ and $\mathcal{Q}[\underline{U}, \underline{W}](s, t) \leq 0$ in an intermediate region (i.e. $\frac{1}{y(t)}<s \leq s_{\star}$).

\begin{lem}\label{lem3.3}
Let $\Omega = B_{R}(0) \subset \mathbb{R}^n$ $(n \geq 3)$. Assume that $D(s)$ and $S(s)$ satisfy $(\ref{eq1.2})$ as well as \eqref{eq1.3} and \eqref{eq1.4} with some $k_D,k_S > 0$ and $p,q\in \mathbb{R}$ fulfilling $(\ref{result-all})$. Suppose that $\delta$, $\alpha$ and $\beta$ are taken from Lemma~\ref{alphabeta}. There exists $s_{\star}=s_{\star}(\alpha,\beta,l,\mu^{\star},\delta)\in (0,R^n)$ small enough such that, if $T>0$ and $y(t) \in C^1([0, T))$ satisfy
\begin{align}\label{y2}
\left\{\begin{array}{l}
0 \leq y^{\prime}(t) \leq y^{1+\delta}(t), \quad t \in (0,T), \\
y(0) \geq y_0
\end{array}\right.
\end{align}
with 
\begin{align}\label{eq3_14}
{y_0}>\frac{1}{s_{\star}},
\end{align}
then, for any $\theta>0$, the functions $\underline{U}$ and $\underline{W}$, defined in $(\ref{eq3_7})$, have the following properties
\begin{align}\label{mid}
\mathcal{P}[\underline{U}, \underline{W}](s, t) \leq 0, \quad  \mathcal{Q}[\underline{U}, \underline{W}](s, t) \leq 0,
\end{align}
for all $t \in (0,T) \cap \left(0, \frac{1}{\theta}\right)$ and $s \in \left( \frac{1}{y(t)},s_{\star}\right]$.
\end{lem}
\begin{proof}
According to \eqref{eq3_14}, we observe that $\frac{1}{y(t)}<\frac{1}{y_0}<s_{\star}$, which ensures that the set $\big(\frac{1}{y(t)},s_{\star}\big]$ is non-empty. Let 
\begin{align}\label{c3}
c_3=\frac{4k_D \mathrm{e}^{|q|+|p|+1}}{c_2^{\beta}k_S}n^{2+p-q}l^{p-q}\alpha^{\frac{2}{n}-1-\alpha+(1-\alpha)(p-q)}\beta^{\beta},\
c_4:=\frac{4\mathrm{e}^{|q|+1}}{c_2^{\beta}k_Sn^{q}}l^{-q}\alpha^{(\alpha-1)q+\delta-\alpha}\beta^{\beta}
\end{align} 
where $\alpha,\beta \in (0,1)$ are taken from Lemma~\ref{alphabeta}. Due to (\ref{alpha})-(\ref{q3}), we can choose $s_{\star}>0$ suitably small such that
\begin{align}\label{sstar1}
{s_{\star}}^{1-\beta} \leq \frac{n\beta^{1-\beta}l}{2\mu^{\star}\mathrm{e}},
\quad{s_{\star}}^{1-\alpha}\leq\frac{n\alpha^{1-\alpha}l}{\mathrm{e}}
\end{align}
and
\begin{align}\label{sstar4}
c_3-s_{\star}^{-\left((1-\alpha)(q-p)+\alpha-\beta-\frac{2}{n}\right)}
\leq 0,
\quad c_4
-s_{\star}^{-\left((1-\alpha)q+\alpha-\beta-\delta\right)}
 \leq 0,
\end{align}
as well as
\begin{align}\label{Q}
{s_{\star}}^{1-\alpha} \leq \frac{n\alpha^{1-\alpha}l}{2\mu^{\star}\mathrm{e}},
\quad
 \frac{4nc_1^{\alpha}\mathrm{e}^{2}\alpha^{\alpha}}{l\beta^{2-\frac{2}{n}} }-s_{\star}^{-(1-\alpha-\frac{2}{n})} \leq 0,
\quad 
\frac{4c_1^{\alpha}\mathrm{e}^{2}\alpha^{\alpha}}{nl}\beta^{\delta-1}-s_{\star}^{-(1-\delta-\alpha)} \leq 0
\end{align}
According to (\ref{eq2.3}) and (\ref{phis})-(\ref{psit}), along with (\ref{thetat}), by a direct computation, we have
\begin{align*}
\begin{aligned}
\mathcal{P}[\underline{U}, \underline{W}](s, t)
&=\underline{U}_t-n^2 s^{2-\frac{2}{n}} D(n\underline{U}_s) \underline{U}_{s s}
-S\left(n \underline{U}_s\right) \cdot\Big(\underline{W}-\frac{\mu^{\star} s}{n}\Big)\\
&=-\theta \mathrm{e}^{-\theta t}\alpha^{-\alpha}l\Big(s-\frac{1-\alpha}{y(t)}\Big)^\alpha
+\mathrm{e}^{-\theta t}\alpha^{1-\alpha} l(1-\alpha)\Big(s-\frac{1-\alpha}{y(t)}\Big)^{\alpha-1}\frac{y^{\prime}(t)}{y^2(t)}\\
&\quad +n^2s^{2-\frac{2}{n}}D\Big(n\mathrm{e}^{-\theta t}\alpha^{1-\alpha} l \Big(s-\frac{1-\alpha}{y(t)}\Big)^{\alpha-1}\Big)\mathrm{e}^{-\theta t}l\alpha^{1-\alpha}(1-\alpha)\Big(s-\frac{1-\alpha}{y(t)}\Big)^{\alpha-2}\\
&\quad -S\left(n\mathrm{e}^{-\theta t}\alpha^{1-\alpha}l\Big(s-\frac{1-\alpha}{y(t)}\Big)^{\alpha-1}\right)
\left(\mathrm{e}^{-\theta t}\beta^{-\beta}l\Big(s-\frac{1-\beta}{y(t)}\Big)^\beta-\frac{\mu^{\star}s}{n}\right)\\
&\leq \alpha^{1-\alpha} l\Big(s-\frac{1-\alpha}{y(t)}\Big)^{\alpha-1}\cdot y^{\delta-1}(t)\\
&+n^2s^{2-\frac{2}{n}}\alpha^{1-\alpha}l \Big(s-\frac{1-\alpha}{y(t)}\Big)^{\alpha-2}D\Big(n\mathrm{e}^{-\theta t}\alpha^{1-\alpha} l \Big(s-\frac{1-\alpha}{y(t)}\Big)^{\alpha-1}\Big)\\
&-S\left(n\mathrm{e}^{-\theta t}\alpha^{1-\alpha}l\Big(s-\frac{1-\alpha}{y(t)}\Big)^{\alpha-1}\right)
\left(\mathrm{e}^{-\theta t}\beta^{-\beta}l\Big(s-\frac{1-\beta}{y(t)}\Big)^\beta-\frac{\mu^{\star}s}{n}\right)
\end{aligned}
\end{align*}
for all $t \in (0,T) \cap \left(0, \frac{1}{\theta}\right)$ and $s \in \left( \frac{1}{y(t)},s_{\star}\right]$. By (\ref{thetat}) and the second restriction in (\ref{sstar1}), we derive that
\begin{align}\label{eq3_17}
n\mathrm{e}^{-\theta t}\alpha^{1-\alpha}l\Big(s-\frac{1-\alpha}{y(t)}\Big)^{\alpha-1}
>\frac{n\alpha^{1-\alpha}l}{ \mathrm{e}} s^{\alpha-1}
\geq\frac{n\alpha^{1-\alpha}l}{ \mathrm{e}} s_{\star}^{\alpha-1}\geq 1.
\end{align}
For $\alpha,\beta \in (0,1)$ taken from Lemma~\ref{alphabeta}, using $s>\frac{1}{y(t)}$, we deduce that
\begin{align*}
\frac{\alpha}{y(t)}<s-\frac{1-\alpha}{y(t)}
, \quad \frac{\beta}{y(t)}<s-\frac{1-\beta}{y(t)}
\end{align*}
and
\begin{align}\label{s1}
\alpha s<s-\frac{1-\alpha}{y(t)},
\quad \beta s<s-\frac{1-\beta}{y(t)}.
\end{align}
Thus, we infer that
\begin{align}\label{s5}
y^{\delta-1}(t)<{\alpha}^{\delta-1}\Big(s-\frac{1-\alpha}{y(t)}\Big)^{1-\delta},
\quad y^{\delta-1}(t)<{\beta}^{\delta-1}\Big(s-\frac{1-\beta}{y(t)}\Big)^{1-\delta}
\end{align}
and
\begin{align}\label{s4}
s^{2-\frac{2}{n}}<\alpha^{\frac{2}{n}-2}{\Big(s-\frac{1-\alpha}{y(t)}\Big)}^{2-\frac{2}{n}},
\quad s^{2-\frac{2}{n}}<\beta^{\frac{2}{n}-2}{\Big(s-\frac{1-\beta}{y(t)}\Big)}^{2-\frac{2}{n}}.
\end{align}
Thus, combining (\ref{eq3_17}) with (\ref{eq1.3}) and (\ref{eq1.4}), applying the first estimates in (\ref{s5}) and (\ref{s4}), utilizing the second estimate in (\ref{s1}), we obtain
\begin{align}\label{eq3_15}
&\mathcal{P}[\underline{U}, \underline{W}](s, t)\nonumber\\
\leq& \alpha^{1-\alpha} l\Big(s-\frac{1-\alpha}{y(t)}\Big)^{\alpha-1}\cdot {\alpha}^{\delta-1}\Big(s-\frac{1-\alpha}{y(t)}\Big)^{1-\delta} \nonumber \\
&+n^2\alpha^{\frac{2}{n}-2}{\Big(s-\frac{1-\alpha}{y(t)}\Big)}^{2-\frac{2}{n}}
\alpha^{1-\alpha}l \Big(s-\frac{1-\alpha}{y(t)}\Big)^{\alpha-2} 
k_D\left(n\mathrm{e}^{-\theta t}\alpha^{1-\alpha}l\Big(s-\frac{1-\alpha}{y(t)}\Big)^{\alpha-1}\right)^p 
 \nonumber \\
& -k_S\left(n\mathrm{e}^{-\theta t}\alpha^{1-\alpha}l\Big(s-\frac{1-\alpha}{y(t)}\Big)^{\alpha-1}\right)^q \left(\mathrm{e}^{-1}\beta^{-\beta}l\Big(s-\frac{1-\beta}{y(t)}\Big)^\beta
-\frac{\mu^{\star}}{n\beta}\Big(s-\frac{1-\beta}{y(t)}\Big)\right) \notag \\
= & \alpha^{\delta-\alpha} l\Big(s-\frac{1-\alpha}{y(t)}\Big)^{\alpha-\delta}  
+k_Dn^{2+p} l^{p+1}\alpha^{\frac{2}{n}-1-\alpha+(1-\alpha)p} \Big(s-\frac{1-\alpha}{y(t)}\Big)^{\alpha-\frac{2}{n}+p(\alpha-1)} (\mathrm{e}^{-\theta t})^p\notag \\
&-k_S\left(n\mathrm{e}^{-\theta t}\alpha^{1-\alpha}l\Big(s-\frac{1-\alpha}{y(t)}\Big)^{\alpha-1}\right)^q \left(\mathrm{e}^{-1}\beta^{-\beta}l\Big(s-\frac{1-\beta}{y(t)}\Big)^\beta
-\frac{\mu^{\star}}{n\beta}\Big(s-\frac{1-\beta}{y(t)}\Big)\right).
\end{align}

We first estimate the last term on the right-hand side of (\ref{eq3_15}). For  $\alpha,\beta \in (0,1)$ taken from Lemma~\ref{alphabeta}, we first prove that
\begin{align}\label{alphabeta-1}
c_1\Big(s-\frac{1-\alpha}{y(t)}\Big) 
\geq s-\frac{1-\beta}{y(t)}
\geq c_2\Big(s-\frac{1-\alpha}{y(t)}\Big),
\end{align}
where $c_1=\max\{\frac{\beta}{\alpha},1\}$ and $c_2=\min\{\frac{\beta}{\alpha},1\}$.
For the case $\alpha \leq \beta$, we have 
\begin{align*}
s-\frac{1-\beta}{y(t)} \geq s-\frac{1-\alpha}{y(t)} \quad \text{and} \quad
s-\frac{1-\alpha}{y(t)} \geq \frac{\alpha}{\beta} \cdot \Big(s-\frac{1-\beta}{y(t)}\Big).
\end{align*}
For the case $\alpha > \beta$, similarly, we know that
\begin{align*}
s-\frac{1-\alpha}{y(t)} \geq s-\frac{1-\beta}{y(t)} \quad \text{and} \quad
s-\frac{1-\beta}{y(t)} \geq \frac{\beta}{\alpha} \cdot \Big(s-\frac{1-\alpha}{y(t)}\Big).
\end{align*}

It follows from (\ref{alphabeta-1}) and the first restriction in (\ref{sstar1}) that
\begin{align*}
&\mathrm{e}^{-1}\beta^{-\beta}l\Big(s-\frac{1-\beta}{y(t)}\Big)^\beta-\frac{\mu^{\star}}{n\beta}\Big(s-\frac{1-\beta}{y(t)}\Big) \nonumber\\
=&\frac{\beta^{-\beta}l}{2\mathrm{e}}\Big(s-\frac{1-\beta}{y(t)}\Big)^\beta
+\Big(s-\frac{1-\beta}{y(t)}\Big)^\beta
\left(\frac{\beta^{-\beta}l}{2\mathrm{e}}-\frac{\mu^{\star}}{n\beta}\Big(s-\frac{1-\beta}{y(t)}\Big)^{1-\beta}\right)\nonumber\\
\geq&\frac{\beta^{-\beta}l}{2\mathrm{e}}\Big(s-\frac{1-\beta}{y(t)}\Big)^\beta
+ \Big(s-\frac{1-\beta}{y(t)}\Big)^\beta
\Big(\frac{\beta^{-\beta}l}{2\mathrm{e}}-\frac{\mu^{\star}}{n\beta}{s}^{1-\beta}\Big)\nonumber\\
\geq &\frac{\beta^{-\beta}l}{2\mathrm{e}}\Big(s-\frac{1-\beta}{y(t)}\Big)^\beta
+\Big(s-\frac{1-\beta}{y(t)}\Big)^\beta
\Big(\frac{\beta^{-\beta}l}{2\mathrm{e}}-\frac{\mu^{\star}}{n\beta}{s_{\star}}^{1-\beta}\Big)\nonumber\\
\geq& \frac{\beta^{-\beta}l}{2\mathrm{e}}\Big(s-\frac{1-\beta}{y(t)}\Big)^\beta 
\geq \frac{c_2^{\beta}\beta^{-\beta}l}{2\mathrm{e}}\Big(s-\frac{1-\alpha}{y(t)}\Big)^\beta .
\end{align*}
Inserting this into \eqref{eq3_15}, along with (\ref{thetat}), we deduce that
\begin{align}\label{p-1}
\mathcal{P}[\underline{U}, \underline{W}](s, t) 
\leq& \alpha^{\delta-\alpha} l\Big(s-\frac{1-\alpha}{y(t)}\Big)^{\alpha-\delta}  
+k_Dn^{2+p} \mathrm{e}^{|p|} l^{p+1}\alpha^{\frac{2}{n}-1-\alpha+(1-\alpha)p} \Big(s-\frac{1-\alpha}{y(t)}\Big)^{\alpha-\frac{2}{n}+p(\alpha-1)}  \notag \\
&-\frac{c_2^{\beta}k_Sn^ql^{q+1}}{2\mathrm{e}^{|q|+1}\beta^{\beta}}\alpha^{(1-\alpha)q}
\Big(s-\frac{1-\alpha}{y(t)}\Big)^{\beta+q(\alpha-1)}.
\end{align}
Finally, we prove $\mathcal{P}[\underline{U}, \underline{W}](s, t) \leq 0$ by making use of the last term on the right-hand side of (\ref{p-1}). By (\ref{q3}), the definition of $c_3$ in (\ref{c3}) and the first restriction in (\ref{sstar4}), after a direct calculation, we derive that
\begin{align}
&k_Dn^{2+p} \mathrm{e}^{|p|} l^{p+1}\alpha^{\frac{2}{n}-1-\alpha+(1-\alpha)p} \Big(s-\frac{1-\alpha}{y(t)}\Big)^{\alpha-\frac{2}{n}+p(\alpha-1)}
-\frac{c_2^{\beta}k_Sn^ql^{q+1}}{4\mathrm{e}^{|q|+1}\beta^{\beta}}\alpha^{(1-\alpha)q}\Big(s-\frac{1-\alpha}{y(t)}\Big)^{\beta+q(\alpha-1)} \notag \\
=&\frac{c_2^{\beta}k_Sn^ql^{q+1}}{4\mathrm{e}^{|q|+1}\beta^{\beta}}\alpha^{(1-\alpha)q}
\Big(s-\frac{1-\alpha}{y(t)}\Big)^{\alpha-\frac{2}{n}+p(\alpha-1)}
\left(c_3-\Big(s-\frac{1-\alpha}{y(t)}\Big)^{-\left((1-\alpha)(q-p)+\alpha-\beta-\frac{2}{n}\right)}\right)
\notag \\
\leq & \frac{c_2^{\beta}k_Sn^ql^{q+1}}{4\mathrm{e}^{|q|+1}\beta^{\beta}}\alpha^{(1-\alpha)q}
\Big(s-\frac{1-\alpha}{y(t)}\Big)^{\alpha-\frac{2}{n}+p(\alpha-1)}
\left(c_3-s^{-\left((1-\alpha)(q-p)+\alpha-\beta-\frac{2}{n}\right)}\right)\notag \\
\leq & \frac{c_2^{\beta}k_Sn^ql^{q+1}}{4\mathrm{e}^{|q|+1}\beta^{\beta}}\alpha^{(1-\alpha)q}
\Big(s-\frac{1-\alpha}{y(t)}\Big)^{\alpha-\frac{2}{n}+p(\alpha-1)}
\left(c_3-s_{\star}^{-\left((1-\alpha)(q-p)+\alpha-\beta-\frac{2}{n}\right)}\right)
\leq 0.
\end{align}
Utilizing the definition of $c_4$ in (\ref{c3}) and the second restriction in (\ref{sstar4}), together with (\ref{q1}), we get
\begin{align}\label{p-2}
&\alpha^{\delta-\alpha} l\Big(s-\frac{1-\alpha}{y(t)}\Big)^{\alpha-\delta}  
-\frac{c_2^{\beta}k_Sn^ql^{q+1}}{4\mathrm{e}^{|q|+1}\beta^{\beta}}\alpha^{(1-\alpha)q}
\Big(s-\frac{1-\alpha}{y(t)}\Big)^{\beta+q(\alpha-1)} \notag \\
\leq&\frac{c_2^{\beta}k_Sn^ql^{q+1}}{4\mathrm{e}^{|q|+1}\beta^{\beta}}\alpha^{(1-\alpha)q}
\Big(s-\frac{1-\alpha}{y(t)}\Big)^{\alpha-\delta}
\left(c_4
-\Big(s-\frac{1-\alpha}{y(t)}\Big)^{-\left((1-\alpha)q+\alpha-\beta-\delta\right)}\right) \notag \\
\leq& \frac{c_2^{\beta}k_Sn^ql^{q+1}}{4\mathrm{e}^{|q|+1}\beta^{\beta}}\alpha^{(1-\alpha)q}
\Big(s-\frac{1-\alpha}{y(t)}\Big)^{\alpha-\delta}
\left(c_4
-s^{-\left((1-\alpha)q+\alpha-\beta-\delta\right)}\right) \notag \\
\leq& \frac{c_2^{\beta}k_Sn^ql^{q+1}}{4\mathrm{e}^{|q|+1}\beta^{\beta}}\alpha^{(1-\alpha)q}
\Big(s-\frac{1-\alpha}{y(t)}\Big)^{\alpha-\delta}
\left(c_4
-s_{\star}^{-\left((1-\alpha)q+\alpha-\beta-\delta\right)}\right) 
\leq 0.
\end{align}
Combining (\ref{p-1})-(\ref{p-2}), we have
\begin{align}\label{p-mid}
\mathcal{P}[\underline{U}, \underline{W}](s, t) 
\leq0,
\end{align}
for all $t \in (0,T) \cap \left(0, \frac{1}{\theta}\right)$ and $s \in \left(\frac{1}{y(t)},s_{\star}\right]$. Similar to (\ref{p-mid}), using the conditions of $s_{\star}$ in (\ref{Q}) and the second estimates in (\ref{s5}) and (\ref{s4}), along with $1-\delta-\alpha>1-\alpha-\frac{2}{n}>0$, which are derived from (\ref{alpha}), we can deduce that
\begin{align}\label{q-mid}
\mathcal{Q}[\underline{U}, \underline{W}](s, t) 
=& \underline{W}_t - n^2 s^{2 - \frac{2}{n}} \underline{W}_{ss} - n \underline{W}_s \cdot \big(\underline{U} - \frac{\mu^{\star} s}{n}\big) \nonumber \\
\leq & \beta^{\delta-\beta} l\Big(s-\frac{1-\beta}{y(t)}\Big)^{\beta-\delta}  
+n^{2}  l\beta^{\frac{2}{n}-1-\beta} \Big(s-\frac{1-\beta}{y(t)}\Big)^{\beta-\frac{2}{n}}  \notag \\
&-\frac{nl^{2}}{2c_1^{\alpha}\mathrm{e}^{2}\alpha^{\alpha}}\beta^{1-\beta}
\Big(s-\frac{1-\beta}{y(t)}\Big)^{\alpha+\beta-1} \nonumber \\
=& \frac{nl^{2}}{4c_1^{\alpha}\mathrm{e}^{2}\alpha^{\alpha}}\beta^{1-\beta} \Big(s-\frac{1-\beta}{y(t)}\Big)^{\beta-\delta}  
\left(\frac{4c_1^{\alpha}\mathrm{e}^{2}\alpha^{\alpha}}{nl}\beta^{\delta-1}-\Big(s-\frac{1-\beta}{y(t)}\Big)^{-(1-\delta-\alpha)}\right) \nonumber \\
&+\frac{nl^{2}}{4c_1^{\alpha}\mathrm{e}^{2}\alpha^{\alpha}}\beta^{1-\beta} \Big(s-\frac{1-\beta}{y(t)}\Big)^{\beta-\frac{2}{n}}  
\left(\frac{4nc_1^{\alpha}\mathrm{e}^{2}\alpha^{\alpha}}{l\beta^{2-\frac{2}{n}} }-\Big(s-\frac{1-\beta}{y(t)}\Big)^{-(1-\alpha-\frac{2}{n})}\right) \nonumber \\
\leq& \frac{nl^{2}}{4c_1^{\alpha}\mathrm{e}^{2}\alpha^{\alpha}}\beta^{1-\beta} \Big(s-\frac{1-\beta}{y(t)}\Big)^{\beta-\delta}  
\left(\frac{4c_1^{\alpha}\mathrm{e}^{2}\alpha^{\alpha}}{nl}\beta^{\delta-1}-s_{\star}^{-(1-\delta-\alpha)}\right) \nonumber \\
&+\frac{nl^{2}}{4c_1^{\alpha}\mathrm{e}^{2}\alpha^{\alpha}}\beta^{1-\beta} \Big(s-\frac{1-\beta}{y(t)}\Big)^{\beta-\frac{2}{n}}  
\left(\frac{4nc_1^{\alpha}\mathrm{e}^{2}\alpha^{\alpha}}{l\beta^{2-\frac{2}{n}} }-s_{\star}^{-(1-\alpha-\frac{2}{n})}\right) \nonumber \\
\leq&0,
\end{align}
for all $t \in (0,T) \cap \left(0, \frac{1}{\theta}\right)$ and $s \in \left(\frac{1}{y(t)},s_{\star}\right]$. Therefore, (\ref{mid}) can be obtained by (\ref{p-mid}) and (\ref{q-mid}). We complete our proof.
\end{proof}
By choosing $\theta>1$ large enough, we prove $\mathcal{P}[\underline{U}, \underline{W}](s, t) \leq 0$ and $\mathcal{Q}[\underline{U}, \underline{W}](s, t) \leq 0$ in the very outer part $s \in  \left(s_{\star}, R^n\right] $.

\begin{lem}\label{lem3.4}
Let $\Omega = B_{R}(0) \subset \mathbb{R}^n$ $(n \geq 3)$.  Assume that $D(s)$ and $S(s)$ satisfy $(\ref{eq1.2})$ as well as \eqref{eq1.3} and \eqref{eq1.4} with some $k_D,k_S > 0$ and $p,q\in \mathbb{R}$ fulfilling $(\ref{result-all})$. Suppose that $\delta$, $\alpha$ and $\beta$ are taken from Lemma~\ref{alphabeta}.
For given $s_{\star}$ taken from Lemma~\ref{lem3.3}, one can find $\theta_{\star}=\theta_{\star}(\alpha,\beta,l,\mu^{\star},s_{\star},\delta)>0$ such that,
if $\theta \geq \theta_{\star}$, $T>0$ and $y(t) \in C^1([0, T))$ satisfy
\begin{align}\label{y2-1}
\left\{\begin{array}{l}
0 \leq y^{\prime}(t) \leq y^{1+\delta}(t), \quad t \in (0,T), \\
y(0) \geq y_0
\end{array}\right.
\end{align}
with 
\begin{align}\label{eq3_14-1}
{y_0}>\frac{1}{s_{\star}},
\end{align}
then the functions $\underline{U}$ and $\underline{W}$, defined in $(\ref{eq3_7})$, have the following properties
\begin{align*}
\mathcal{P}[\underline{U}, \underline{W}](s, t) \leq 0, \quad \mathcal{Q}[\underline{U}, \underline{W}](s, t) \leq 0,
\end{align*}
for all $t \in (0,T) \cap \left(0, \frac{1}{\theta}\right)$ and $s \in \left(s_{\star}, R^n\right)$.
\end{lem}
\begin{proof}
Due to (\ref{eq3_14-1}) and $y^{\prime}(t) \geq 0$, we have 
\begin{align}\label{s7}
R^{n}>s-\frac{1-\alpha}{y(t)}>s_{\star}-\frac{1-\alpha}{y(t)}>\alpha s_{\star} \ \text{ and } \
 R^{n}>s-\frac{1-\beta}{y(t)}>s_{\star}-\frac{1-\beta}{y(t)}>\beta s_{\star}.
\end{align}
Thus, using (\ref{thetat}), we have 
\begin{align}\label{s10}
n \mathrm{e}^{-\theta t}l \alpha^{1-\alpha}\Big(s-\frac{1-\alpha}{y(t)}\Big) ^{\alpha-1}
\in \left[nl \mathrm{e}^{-1} \alpha^{1 - \alpha} R^{n(\alpha - 1)}, nl s^{\alpha - 1}_{\star} \right].
\end{align}
Therefore, based on the continuity of $D$ and $S$, we define
\begin{align*}
	D_{\max} := \max\left\{ D(x) \mid x \in \left[nl\mathrm{e}^{-1} \alpha^{1 - \alpha} R^{n(\alpha - 1)}, nl   s^{\alpha - 1}_{\star} \right]\right\}
\end{align*}
and
\begin{align*}
	S_{\max} := \max\left\{ S(x) \mid x \in \left[nl\mathrm{e}^{-1} \alpha^{1 - \alpha} R^{n(\alpha - 1)}, nl   s^{\alpha - 1}_{\star} \right]\right\}.
\end{align*} 
Selecting $\theta_{\star}$ large enough to satisfy
\begin{align}\label{eq3_18}
	\theta_{\star} \cdot \frac{l {s_{\star}}^\alpha}{\mathrm{e}}
\geq ls^{\alpha-\delta}_{\star} + \frac{n^2 R^{2n - 2} l  {s_{\star}}^{\alpha - 2} D_{\max} }{\alpha}  + \frac{\mu^{\star} R^n{S_{\max}}}{n}
\end{align}
and 
\begin{align}\label{eq3_18-1}
	\theta_{\star} \cdot \frac{l {s_{\star}}^\beta}{\mathrm{e}} \geq l s^{\beta-\delta}_{\star}+ \frac{n^2 R^{2n - 2}l {s_{\star}}^{\beta - 2}}{\beta} \cdot  + \frac{\mu^{\star} R^n }{n}.
\end{align}
Again using the fact that $y(t)>y_0>\frac{1}{s_{\star}}$, we derive that $y^{\delta-1}<s^{1-\delta}_{\star}$. Combining this with the first restriction in (\ref{s7}) and $\alpha<1$, we obtain
\begin{align}\label{out}
\mathrm{e}^{-\theta t} l \alpha^{1-\alpha} (1-\alpha) \Big(s-\frac{1-\alpha}{y(t)}\Big)^{\alpha-1}y^{\delta-1}(t)
\leq  l \alpha^{1-\alpha} (\alpha s_{\star})^{\alpha-1} s^{1-\delta}_{\star}
=ls^{\alpha-\delta}_{\star}.
\end{align}

Thus, applying (\ref{thetat}), (\ref{y2-1}), (\ref{s10}), (\ref{eq3_18}), (\ref{out}) and the first inequality in (\ref{s7}), after a direct computation, we obtain
\begin{align*}
\begin{aligned}
\mathcal{P}[\underline{U}, \underline{W}](s, t) =& \underline{U}_t - n^2 s^{2 - \frac{2}{n}}D(n\underline{U}_{s}) \underline{U}_{ss} - S\left(n \underline{U}_s\right) \cdot \big(\underline{W} - \frac{\mu^{\star} s}{n}\big) \\
=& -\theta \mathrm{e}^{-\theta t} l \alpha^{-\alpha} \Big(s-\frac{1-\alpha}{y(t)}\Big) ^{\alpha} 
+ \mathrm{e}^{-\theta t} l \alpha^{1-\alpha} (1-\alpha) \Big(s-\frac{1-\alpha}{y(t)}\Big)^{\alpha-1}\frac{y^{\prime}(t)}{y^2(t)}\\ \notag
&+ n^2 s^{2-\frac{2}{n}} \mathrm{e}^{-\theta t}l(1-\alpha) \alpha^{1-\alpha}  \Big(s-\frac{1-\alpha}{y(t)}\Big)^{\alpha-2} D\Big(n\mathrm{e}^{-\theta t}l \alpha^{1-\alpha}\Big(s-\frac{1-\alpha}{y(t)}\Big)^{\alpha-1} \Big)\\ \notag
&-S\Big(n\mathrm{e}^{-\theta t}l \alpha^{1-\alpha}\Big(s-\frac{1-\alpha}{y(t)}\Big)^{\alpha-1} \Big) \cdot \Big(\mathrm{e}^{-\theta t}l \beta^{-\beta} \Big(s-\frac{1-\beta}{y(t)}\Big)^{\beta}-\frac{\mu^{\star}s}{n}\Big)\\ \notag
\leq & -\frac{l\theta \alpha^{-\alpha}}{\mathrm{e}} {(\alpha s_{\star})}^\alpha+  l  s^{\alpha-\delta}_{\star} + n^2 R^{2n - 2} l \alpha^{1-\alpha}  {(\alpha s_{\star})}^{\alpha - 2} D_{\max}   + \frac{\mu^{\star} R^n{S_{\max}}}{n} \\ \notag
\leq & -\frac{l\theta_{\star} {s_{\star}}^\alpha}{\mathrm{e}}+  ls^{\alpha-\delta}_{\star} + \frac{n^2 R^{2n - 2} l  {s_{\star}}^{\alpha - 2} D_{\max} }{\alpha}  + \frac{\mu^{\star} R^n{S_{\max}}}{n}
\leq  0.
\end{aligned}
\end{align*}
Similarly, due to (\ref{eq3_18-1}) and the second restriction in (\ref{s7}), we have
\begin{align*}
\begin{aligned}
\mathcal{Q}[\underline{U}, \underline{W}](s, t) 
=& \underline{W}_t - n^2 s^{2 - \frac{2}{n}} \underline{W}_{ss} - n \underline{W}_s \cdot \big(\underline{U} - \frac{\mu^{\star} s}{n}\big) \\
\leq & -\frac{l\theta {s_{\star}}^\beta}{\mathrm{e}}+  ls^{\beta-\delta}_{\star} + \frac{n^2 R^{2n - 2}l {s_{\star}}^{\beta - 2}}{\beta}  + \frac{\mu^{\star} R^n}{n}
\leq  0.
\end{aligned}
\end{align*}
We complete our proof.
\end{proof}

\emph{\textbf{Proof of Theorem \ref{thm1_1}.}}
Let $\delta$, $\alpha$ and $\beta$ be as in Lemma~\ref{alphabeta}.
For given $s_{\star} \in (0,R^n)$ satisfying (\ref{sstar1})-(\ref{Q}) and $\theta_{\star}$ from Lemma~\ref{lem3.4}, we can fix $\theta$ such that
\begin{align}\label{theta-1}
\theta>\theta_{\star}.
\end{align}
Let $y$ be the solution of the ordinary differential equation
\begin{align}\label{eq3.2}
	\begin{cases}
		y^{\prime}(t) =\kappa y^{1 + \delta}(t), \quad t \in (0, T), \\
		y(0) = y_0,
	\end{cases}
\end{align}
where 
\begin{align}\label{y0}
\kappa= \min\Big\{1,\frac{k_Sn^ql^{q}}{2\mathrm{e}^{|q|+1}}  , \frac{ nl}{2\mathrm{e}^{2}} \Big\}, \quad  y_0>\max\big\{\big(\frac{\theta}{\kappa \delta }\big)^{\frac{1}{\delta}}, y_{\star},\frac{1}{s_{\star}}\big\}
\end{align}
with $y_{\star}$ defined in (\ref{ystar}) and
\begin{align*}
	T = \frac{1}{\kappa \delta} y_0^{-\delta}.
\end{align*}
Then,
\begin{align}\label{T-1}
T<\frac{1}{\theta}
\end{align}
and
\begin{align*}
\lim _{t \nearrow T_{\max }}y(t)=\infty.
\end{align*}
(\ref{eq3.2}) and (\ref{y0}) ensure that all the requirements on $y(t)$ in Lemmas ~\ref{lem3.1}-\ref{lem3.4} are met, while (\ref{theta-1}) ensures the requirement on $\theta$ is satisfied. Therefore, we may apply Lemmas~\ref{lem3.1}-\ref{lem3.4}, along with (\ref{T-1}), to conclude that
\begin{align}\label{all}
\mathcal{P}[\underline{U}, \underline{W}](s, t) \leq 0, 
\quad \mathcal{Q}[\underline{U}, \underline{W}](s, t) \leq 0,
\end{align}
for all $t \in (0,T) $ and $s \in \left(0, R^n\right) \backslash \{\frac{1}{y(t)}\}$.

By Lemma~\ref{s}, (\ref{eq2_5}) and (\ref{all}), we can use the comparison principle from Lemma~\ref{lem2.2} to verify that 
\begin{align*}
\underline{U}(s, t) \leq U(s, t) , 
\quad \underline{W}(s, t) \leq W(s, t),
\quad  (s,t) \in [0, R^n]   \times [0, T).
\end{align*}
In view of $U(0,t)=\underline{U}(0,t)=0$, we have
\begin{align}\label{us=0}
\frac{1}{n} \cdot u(0, t) = U_s(0, t) \geq \underline{U}_s(0, t) = \mathrm{e}^{-\theta t} \cdot l y^{1 - \alpha}(t) \geq \frac{l}{\mathrm{e}} \cdot y^{1 - \alpha}(t) \rightarrow +\infty \quad \text{as } t \nearrow T.
\end{align}
Similarly, we conclude that
\begin{align}\label{ws=0}
\frac{1}{n} \cdot w(0, t)  \geq \frac{l}{\mathrm{e}} \cdot y^{1 - \beta}(t) \rightarrow +\infty \quad \text{as } t \nearrow T.
\end{align}
Therefore, (\ref{us=0}) and (\ref{ws=0}) ensure $T_{\max} \leq T < \infty$.

\section{Global Boundedness. Proof of Theorem \ref{thm1_2}}\label{bounded}
In this section, we prove Theorem~\ref{thm1_2}. Inspired by \cite{2021-JMAA-Zhong}, we construct a Lyapunov functional of the form 
\begin{align}\label{f-def}
\mathcal{F}(t):=\mathcal{F}(u, v, w, z)=\int_{\Omega} {G}(u) d x+\int_{\Omega} w\mathrm{ln}w d x-\int_{\Omega} \nabla v \cdot \nabla z d x,
\end{align}
where ${G}$ is given by 
\begin{align}\label{G-1}
{G}(s):=\int_{s_1}^s \int_{s_1}^\sigma \frac{{D}(\tau)}{S(\tau)} d \tau d \sigma
\end{align}
with some $s>s_1>0$. The Lyapunov functional $\mathcal{F}(t)$ is nonincreasing along trajectories and plays a crucial role in demonstrating the boundedness of solutions.
\begin{lem}\label{lemma-f}
Suppose that $\left(u, v, w, z\right)$ is a classical solution of $(\ref{eq1.1.0})$, as given in Proposition~\ref{local}. Assume that the functions $D(s)$ and $S(s)$ fulfill $(\ref{eq1.2})$. Then
\begin{align}\label{f}
\frac{d}{d t} \mathcal{F}(t)
=-\int_{\Omega} S(u)\left|\frac{\mathcal{D}(u)}{S(u)} \nabla u-\nabla v\right|^2 d x-\int_{\Omega} w\left|\frac{\nabla w}{w} -\nabla z\right|^2 d x
=-\mathcal{D}(t), \quad t>0 .
\end{align}
Moreover, let $u_0, w_0$ satisfy $(\ref{eq1.1})$, and let $v_0, z_0 \in C^2(\bar{\Omega})$ be the corresponding solutions of
\begin{align}\label{v0}
0=\Delta v_0 -\mu_w +w_0, \quad \left.\partial_\nu v_0\right|_{\partial \Omega}=0, \quad \int_{\Omega} v_0 dx=0
\end{align} 
and 
\begin{align}\label{z0}
0=\Delta z_0 -\mu_u +u_0, \quad \left.\partial_\nu z_0\right|_{\partial \Omega}=0, 
\quad \int_{\Omega} z_0 dx=0,
\end{align}
then
\begin{align}\label{f-inter}
\mathcal{F}(t)-\mathcal{F}(0)=-\int_0^t \mathcal{D}(\zeta) d \zeta, \quad t \in\left(0, T_{\max }\right).
\end{align}
\end{lem}
\begin{proof}
Multiplying the first equation in (\ref{eq1.1.0}) by $G^{\prime}(u)$, we deduce that
\begin{align*}
&\frac{d}{d t} \int_{\Omega} {G}(u) dx \nonumber \\
=& \int_{\Omega}G^{\prime}(u) \nabla \cdot(D(u) \nabla u-S(u) \nabla v) dx \nonumber \\ 
=&-\int_{\Omega} \frac{D^2(u)}{S(u)}|\nabla u|^2 dx
+ \int_{\Omega}D(u) \nabla u \cdot \nabla v dx  \nonumber \\
=&- \int_{\Omega} S(u) \left|\frac{D(u)}{S(u)}\nabla u -\nabla v\right|^2 d x
+ \int_{\Omega} S(u)  |\nabla v|^2 dx
- \int_{\Omega} D(u)\nabla u \cdot \nabla v dx .
\end{align*}
Using the first and fourth equations in (\ref{eq1.1.0}) and integrating by parts, we have
\begin{align*}
- \int_{\Omega} D(u)\nabla u \cdot \nabla v dx
=& \int_{\Omega} \Big(u_t+\nabla \cdot \big(S(u) \nabla v\big)\Big) v dx \nonumber  \\ 
=& \int_{\Omega} u_t v  dx
-  \int_{\Omega} S(u) |\nabla v|^2  dx \nonumber  \\ 
=& \int_{\Omega} (\mu_u-\Delta z)_t v  dx
-  \int_{\Omega} S(u) |\nabla v|^2  dx \nonumber  \\ 
=& \int_{\Omega} (\nabla z)_t \nabla v  dx
-  \int_{\Omega} S(u) |\nabla v|^2  dx .
\end{align*}
Thus, we deduce that
\begin{align}\label{u-f}
\frac{d}{d t} \int_{\Omega} {G}(u) dx
=- \int_{\Omega} S(u) \left|\frac{D(u)}{S(u)}\nabla u -\nabla v\right|^2 d x
+\int_{\Omega} (\nabla z)_t \cdot \nabla v  dx.
\end{align}
By the same way, we know that
\begin{align}\label{w-f}
\frac{d}{d t} \int_{\Omega}w \mathrm{ln}w dx
=- \int_{\Omega} w \left|\frac{\nabla w}{w} -\nabla z\right|^2 d x
+\int_{\Omega} \nabla z \cdot (\nabla v)_t  dx.
\end{align}
Consequently, by adding (\ref{u-f}) and (\ref{w-f}), we obtain (\ref{f}). For any $\epsilon>0$, we integrate (\ref{f}) over $(\epsilon,t)$. Thanks to Proposition~\ref{local}, we are able to take $\epsilon \rightarrow 0$ and deduce (\ref{f-inter}). 
\end{proof}
By employing the fact that the Lyapunov functional is nonincreasing along trajectories, we first obtain the boundedness of $\|u\|_{L^{p-q+2}(\Omega)}$ when $q-p<2-\frac{n}{2}$.
\begin{lem}
Let $n \geq 3$. Assume that $D(s)$ and $S(s)$ satisfy $(\ref{eq1.2})$ as well as $(\ref{D-g})$ and $(\ref{S-g})$ with some $K_D, K_S>0$ and $p,q\in \mathbb{R}$ fulfilling 
\begin{align*}
q-p<2-\frac{n}{2}.
\end{align*}
Then, for any $(u_0,v_0,w_0,z_0)$ satisfying $(\ref{eq1.1})$, $(\ref{v0})$ and $(\ref{z0})$, there exists a positive constant $C$, independent of $t$, such that
\begin{align}\label{up-q+2}
 \int_{\Omega} (1+u)^{p-q+2} d x \leq C ,
   \quad t\in (0,T_{\max}).
\end{align}
\end{lem}
\begin{proof}
According to (\ref{f-inter}) and $\int_{\Omega} z(\cdot,t) dx=0$, we deduce that
\begin{align}\label{f-in}
\int_{\Omega}{G}(u) d x+\int_{\Omega} w\mathrm{ln}w d x 
& \leq \int_{\Omega} \nabla v \cdot \nabla z dx+\mathcal{F}(0) \nonumber  \\ 
& =-\int_{\Omega}(\Delta v) z d x+\mathcal{F}(0) \nonumber  \\ 
& =\int_{\Omega} w z d x+\mathcal{F}(0).
\end{align}
It follows from (\ref{D-g}) and (\ref{S-g}) that 
\begin{align*}
{G}(s):=&\int_{s_0}^s \int_{s_0}^\sigma \frac{{D}(\tau)}{S(\tau)} d \tau d \sigma
\geq \frac{K_D}{K_S}\int_{s_0}^s \int_{s_0}^\sigma (1+\tau)^{p-q} d \tau d \sigma \nonumber \\
=&\frac{K_D}{K_S}\Big(\frac{(1+s)^{p-q+2}}{(p-q+1)(p-q+2)}-\frac{(1+s_0)^{p-q+2}}{(p-q+1)(p-q+2)}-\frac{(1+s_0)^{p-q+1}(s-s_0)}{p-q+1}\Big).
\end{align*}
Integrating with respect to $x$ and using $\int_{\Omega}u(\cdot,t) d x=\int_{\Omega}u_0 d x$, along with $q-p<2-\frac{n}{2}<1$ by $n \geq 3$, one can find a positive constant $c_1$ such that
\begin{align*}
\int_{\Omega} (1+u)^{p-q+2} dx
\leq c_1\int_{\Omega}G(u) dx +c_1.
\end{align*}
Substituting this into (\ref{f-in}), there exists a positive constant $c_2$ such that
\begin{align}\label{f-all}
\int_{\Omega} (1+u)^{p-q+2} dx
+c_1 \int_{\Omega} w\mathrm{ln}w d x 
\leq &c_1\int_{\Omega} w z d x+c_2 \nonumber \\
\leq &c_1\| w_0\|_{L^1(\Omega)}\|z\|_{L^{\infty}(\Omega)}+c_2.
\end{align}
For a fixed $b$ satisfying $p-q+2>b>\frac{n}{2}>1$, one can apply standard elliptic regularity theory, Sobolev imbedding theorems and Young's inequality to obtain positive constants $c_3, c_4, c_5>0$ such that
\begin{align}\label{zLinfty}
\|z\|_{L^{\infty}(\Omega)}
\leq c_3\|z\|_{W^{2,b}(\Omega)}
 \leq c_4 \|u\|_{L^{b}(\Omega)} 
\leq \frac{1}{2c_1\| w_0\|_{L^1(\Omega)}}\int_{\Omega} (1+u)^{p-q+2} dx +c_5.
\end{align}
Inserting (\ref{zLinfty}) into (\ref{f-all}), along with $x\mathrm{ln}x \geq -\frac{1}{e} $ for all $x>0$, we complete our proof.
\end{proof}
\begin{lem}
Let $n \geq 3$. Assume that $D(s)$ and $S(s)$ satisfy $(\ref{eq1.2})$ as well as $(\ref{D-g})$ and $(\ref{S-g})$ with some $K_D, K_S>0$ and $p,q\in \mathbb{R}$ fulfilling 
\begin{align}\label{pq-all-1}
q-p<2-\frac{n}{2}.
\end{align}
Then for all $k_1, k_2 \in(1,\infty)$, there exists a positive constant $C(k_1,k_2)$, independent of $t$, such that
\begin{align}\label{uw-bounded}
  \|u(\cdot, t)\|_{L^{k_1}\left(\Omega \right)}+\ \|w(\cdot, t)\|_{L^{k_2}\left(\Omega \right)} \leq C(k_1,k_2), 
  \quad  t\in (0,T_{\max}).
\end{align}
\end{lem}
\begin{proof}
We first claim that, for any
\begin{align*}
k_1>\max\Big\{1,\frac{(1+\frac{2}{n})(p-q+2)+p}{\frac{2}{n}(p-q+2)+p-q+1}-\frac{2}{n}\Big\},
\end{align*}
we can choose 
\begin{align}\label{k2}
k_2 \in \Big(&\max\big\{p-q+2,\frac{n}{2}\big(k_1+\frac{2}{n}\big)-\frac{2}{n}(p-q+2)-p\big\}, \\ \notag
&\big(\frac{2}{n}(p-q+2)+p-q+1\big)\big(k_1+\frac{2}{n}\big)-\frac{2}{n}(p-q+2)-p\Big).
\end{align}
Owing to (\ref{pq-all-1}), we infer that 
\begin{align*}
\frac{2}{n}(p-q+2)+p-q+1>\frac{2}{n}\cdot \frac{n}{2}+\frac{n}{2}-1=\frac{n}{2}.
\end{align*}
Then, for any $k_1>1$, we have
\begin{align}\label{k2-1}
&\frac{n}{2}\big(k_1+\frac{2}{n}\big)-\frac{2}{n}(p-q+2)-p \notag\\ 
<&\Big(\frac{2}{n}(p-q+2)+p-q+1\Big)\big(k_1+\frac{2}{n}\big)-\frac{2}{n}(p-q+2)-p.
\end{align}
Thanks to $k_1>\frac{(1+\frac{2}{n})(p-q+2)+p}{\frac{2}{n}(p-q+2)+p-q+1}-\frac{2}{n}$, we know that
\begin{align*}
p-q+2<\Big(\frac{2}{n}(p-q+2)+p-q+1\Big)\big(k_1+\frac{2}{n}\big)-\frac{2}{n}(p-q+2)-p.
\end{align*}
Combining this with (\ref{k2-1}) guarantees the existence of $k_2$.

Since $k_2$ satisfies
\begin{align*}
\frac{n}{2}\big(k_1+\frac{2}{n}\big)-\frac{2}{n}(p-q+2)-p<k_2<\Big(\frac{2}{n}(p-q+2)+p-q+1\Big)\big(k_1+\frac{2}{n}\big)-\frac{2}{n}(p-q+2)-p,
\end{align*}
we obtain that
\begin{align}\label{pq}
\frac{n}{2}\big(k_1+\frac{2}{n}\big)<\frac{2}{n}(p-q+2)+k_2+p<\Big(\frac{2}{n}(p-q+2)+p-q+1\Big)\big(k_1+\frac{2}{n}\big).
\end{align}
Using the right hand side of (\ref{pq}), we derive that
\begin{align}\label{eq-1}
\frac{(k_1+\frac{2}{n})(k_2+q-1)}{k_1+\frac{2}{n}-1}<\frac{2}{n}(p-q+2)+k_2+p.
\end{align}
Since $k_2>(1-\frac{2}{n})(p-q+2)-p$ by $k_2>p-q+2$, we deduce that
\begin{align}\label{k1} \frac{1}{2}-\frac{1}{n}<\frac{k_2+p}{2(\frac{2}{n}(p-q+2)+k_2+p)}<\frac{k_2+p}{2(p-q+2)}.
\end{align}
Again using $k_2>p-q+2$, we have
\begin{align}\label{k1-2} 
\frac{1}{2}-\frac{1}{n}<\frac{k_2+p}{2k_2}<\frac{k_2+p}{2(p-q+2)}.
\end{align}

By means of (\ref{S-g}) and $k_2>p-q+2>\frac{n}{2}$, we have
\begin{align}\label{F}
F(u):=\int_0^u S(\sigma)(1+\sigma)^{k_2-2}  d\sigma
\leq K_S\int_0^u (1+\sigma)^{k_2+q-2}  d\sigma
\leq \frac{K_S (1+u)^{k_2+q-1}}{k_2+q-1} .
\end{align}
Multiplying the first equation in (\ref{eq1.1.0}) by $(1+u)^{k_2-1}$ and integrating by parts, along with (\ref{F}) and (\ref{D-g}), we find that
\begin{align}\label{uk1}
&\frac{1}{{k_2}} \frac{d}{d t} \int_{\Omega} (1+u)^{k_2} dx \nonumber \\
=&-({k_2}-1)\int_{\Omega} D(u)(1+u)^{{k_2}-2}|\nabla u|^2 dx 
+({k_2}-1)\int_{\Omega} S(u)(1+u)^{{k_2}-2} \nabla u \cdot \nabla v dx\notag \\
\leq &  -K_D({k_2}-1) \int_{\Omega}(1+u)^{{k_2}+p-2}|\nabla u|^2 dx
+({k_2}-1) \int_{\Omega} S(u)(1+u)^{{k_2}-2} \nabla u \cdot \nabla v dx\notag \\
= & \frac{-4K_D({k_2}-1)}{({k_2}+p)^2} \int_{\Omega} |\nabla (1+u)^\frac{{k_2}+p}{2}|^2dx -({k_2}-1)\int_{\Omega}F(u) \Delta vdx  \notag \\
\leq & \frac{-4K_D({k_2}-1)}{({k_2}+p)^2} \int_{\Omega} |\nabla (1+u)^\frac{{k_2}+p}{2}|^2dx 
+({k_2}-1) \int_{\Omega}F(u) wdx \notag \\
\leq & \frac{-4K_D({k_2}-1)}{({k_2}+p)^2} \int_{\Omega} |\nabla (1+u)^\frac{{k_2}+p}{2}|^2dx 
+\frac{K_S({k_2}-1)}{{k_2}+q-1} \int_{\Omega}(1+u)^{{k_2}+q-1} wdx .
\end{align}
Similarly, multiplying the third equation in (\ref{eq1.1.0}) by $w^{{k_1}-1}$ and integrating by parts, we infer that
\begin{align}\label{wk2}
\frac{1}{{k_1}} \frac{d}{d t} \int_{\Omega} w^{k_1} dx 
+ \frac{4({k_1}-1)}{{k_1}^2} \int_{\Omega} |\nabla w^{\frac{{k_1}}{2}}|^2dx  
& \leq \frac{{k_1}-1}{{k_1}}  \int_{\Omega} w^{k_1} (1+u) dx.
\end{align}
By the Gagliardo-Nirenberg inequality, we can find a constant $c_1=c_1(k_1)>0$ such that
\begin{align}\label{G-N}
\int_{\Omega} w^{k_1+\frac{2}{n}}  d x 
\leq & c_1 \cdot\left(\int_{\Omega} w d x\right)^{(k_1+\frac{2}{n})(1-\theta_1)} \cdot 
\int_{\Omega}\left|\nabla w^{\frac{k_1}{2}}\right|^2  d x
+c_1 \cdot\left(\int_{\Omega} w  d x\right)^{k_1+\frac{2}{n}} 
\end{align}
 where $\theta_1=\frac{\frac{k_1}{2}-\frac{k_1}{2(k_1+\frac{2}{n})}}{\frac{1}{n}+\frac{k_1}{2}-\frac{1}{2}}
\in (0,1)$ due to $k_1>1$. Using Young's inequality, (\ref{G-N}) and the mass conservation property of $w(x,t)$, we obtain the positive constants $c_2=c_2(k_1,k_2)$, $c_3=c_3(k_1,k_2)$ and $c_4=c_4(k_1,k_2)$ such that
\begin{align}\label{uk1+q-1w}
&\frac{K_S({k_2}-1)}{{k_2}+q-1} \int_{\Omega}(1+u)^{{k_2}+q-1} wdx \nonumber \\
\leq& c_2 \int_{\Omega} w^{k_1+\frac{2}{n}}  d x 
+c_3 \int_{\Omega} (1+u)^{\frac{(k_1+\frac{2}{n})(k_2+q-1)}{k_1+\frac{2}{n}-1}}  d x \nonumber \\
\leq& \frac{{k_1}-1}{{k_1}^2} \int_{\Omega} |\nabla w^{\frac{{k_1}}{2}}|^2dx 
+ c_3 \int_{\Omega} (1+u)^{\frac{(k_1+\frac{2}{n})(k_2+q-1)}{k_1+\frac{2}{n}-1}}  d x
+c_4.
\end{align}
Similar to (\ref{uk1+q-1w}), we derive that
\begin{align*}
\frac{{k_1}-1}{{k_1}}  \int_{\Omega} w^{k_1} (1+u) dx
\leq& \frac{{k_1}-1}{{k_1}^2} \int_{\Omega} |\nabla w^{\frac{{k_1}}{2}}|^2dx 
+ c_3 \int_{\Omega} (1+u)^{\frac{k_1+\frac{2}{n}}{\frac{2}{n}}}  d x
+c_4.
\end{align*}
Combining this with (\ref{uk1}), (\ref{wk2}) and (\ref{uk1+q-1w}), we infer that
\begin{align}\label{uk1wk2-1}
 &\frac{d}{d t} \left(\frac{1}{{k_2}}\int_{\Omega} (1+u)^{k_2} dx 
 + \frac{1}{{k_1}} \int_{\Omega} w^{k_1} dx \right)
+\frac{4K_D({k_2}-1)}{({k_2}+p)^2} \int_{\Omega} |\nabla (1+u)^\frac{{k_2}+p}{2}|^2dx \nonumber \\
&+\frac{2({k_1}-1)}{{k_1}^2} \int_{\Omega} |\nabla w^{\frac{{k_1}}{2}}|^2dx \nonumber \\
\leq& c_3 \int_{\Omega} (1+u)^{\frac{(k_1+\frac{2}{n})(k_2+q-1)}{k_1+\frac{2}{n}-1}}  d x
+c_3 \int_{\Omega} (1+u)^{\frac{n}{2}(k_1+\frac{2}{n})}  d x
+2c_4.
\end{align}
Again applying the Gagliardo-Nirenberg inequality, there exists a positive constant $c_5=c_5(k_2)$ such that
\begin{align}\label{G-N-3}
\int_{\Omega} (1+u)^{\frac{2}{n}(p-q+2)+k_2+p}d x
\leq& c_5\Big(\int_{\Omega} (1+u)^{p-q+2} d x\Big)^{\frac{\frac{2}{n}(p-q+2)+k_2+p}{p-q+2}(1-\theta_2)} 
\int_{\Omega}\left|\nabla (1+u)^{\frac{k_2+p}{2}}\right|^2  d x \nonumber \\
&+c_5\Big(\int_{\Omega}(1+u)^{p-q+2} d x\Big)^{\frac{2}{n}+\frac{k_2+p}{p-q+2}}
\end{align}
with $\theta_2=\frac{\frac{k_2+p}{2(p-q+2)}-\frac{k_2+p}{2(\frac{2}{n}(p-q+2)+k_2+p)}}{\frac{1}{n}+\frac{k_2+p}{2(p-q+2)}-\frac{1}{2}}\in (0,1)$ by (\ref{k1}). Using (\ref{eq-1}) and Young's inequality, along with (\ref{up-q+2}) and (\ref{G-N-3}), one can find positive constants $c_6=c_6(k_1,k_2)$, $c_7=c_7(k_1,k_2)$ and $c_8=c_8(k_1,k_2)$ such that
\begin{align}\label{u-complex}
c_3 \int_{\Omega} (1+u)^{\frac{(k_1+\frac{2}{n})(k_2+q-1)}{k_1+\frac{2}{n}-1}}  d x
\leq & c_6 \int_{\Omega} (1+u)^{\frac{2}{n}(p-q+2)+k_2+p}d x +c_7 \nonumber \\
\leq & \frac{K_D({k_2}-1)}{({k_2}+p)^2} \int_{\Omega} |\nabla (1+u)^\frac{{k_2}+p}{2}|^2dx  +c_8.
\end{align}
Similarly, according to the first inequality in (\ref{pq}), we deduce that
\begin{align}\label{u-easy}
c_3 \int_{\Omega} (1+u)^{\frac{n}{2}(k_1+\frac{2}{n})}  d x
\leq  \frac{K_D({k_2}-1)}{({k_2}+p)^2} \int_{\Omega} |\nabla (1+u)^\frac{{k_2}+p}{2}|^2dx  +c_8.
\end{align}
Inserting (\ref{u-complex}) and (\ref{u-easy}) into (\ref{uk1wk2-1}), we have
\begin{align*}
 &\frac{d}{d t} \left(\frac{1}{{k_2}}\int_{\Omega} (1+u)^{k_2} dx 
 + \frac{1}{{k_1}} \int_{\Omega} w^{k_1} dx \right)
+\frac{2K_D({k_2}-1)}{({k_2}+p)^2} \int_{\Omega} |\nabla (1+u)^\frac{{k_2}+p}{2}|^2dx \nonumber \\
+&\frac{2({k_1}-1)}{{k_1}^2} \int_{\Omega} |\nabla w^{\frac{{k_1}}{2}}|^2dx
\leq 2c_8
+2c_4.
\end{align*}
Applying the Gagliardo-Nirenberg inequality, and using (\ref{up-q+2}), we infer the existence of positive constants $c_9=c_9(k_2)$ and $c_{10}=c_{10}(k_2)$ such that
\begin{align}\label{u-G-N}
\|(1+u)^{\frac{k_2+p}{2}}\|_{L^{\frac{2k_2}{k_2+p}}}^{r_1}
\leq& c_9\|\nabla(1+u)^{\frac{k_2+p}{2}}\|_{L^2}^{2}
\|(1+u)^{\frac{k_2+p}{2}}\|_{L^{\frac{2(p-q+2)}{k_2+p}}}^{r_1(1-\theta_3)}
+c_9 \|(1+u)^{\frac{k_2+p}{2}}\|_{L^{\frac{2(p-q+2)}{k_2+p}}}^{r_1}
\nonumber \\
\leq & c_{10} \|\nabla(1+u)^{\frac{k_2+p}{2}}\|_{L^2}^{2} +c_{10}.
\end{align}
where $\theta_3=\frac{\frac{k_2+p}{2(p-q+2)}-\frac{k_2+p}{2k_2}}{\frac{1}{n}+\frac{k_2+p}{2(p-q+2)}-\frac{1}{2}} \in (0,1)$ and $r_1=\frac{2}{\theta_3}=\frac{\frac{2}{n}+\frac{k_2+p}{p-q+2}-1}{\frac{k_2    +p}{2(p-q+2)}-\frac{k_2+p}{2k_2}}>0$ by (\ref{k1-2}). 
We use the Gagliardo-Nirenberg inequality again to show the existence of constant $c_{11}=c_{11}(k_1)$ such that
\begin{align}\label{w-G-N}
\|w^{\frac{k_1}{2}}\|_{L^{2}}^{r_2}
\leq c_{11}\|\nabla w^{\frac{k_1}{2}}\|_{L^2}^{2}
\|w^{\frac{k_1}{2}}\|_{L^{\frac{2}{k_1}}}^{r_2(1-\theta_4)}
+c_{11} \|w^{\frac{k_1}{2}}\|_{L^{\frac{2}{k_1}}}^{r_2},
\end{align}
where $\theta_4=\frac{\frac{k_1}{2}-\frac{1}{2}}{\frac{1}{n}+\frac{k_1}{2}-\frac{1}{2}} \in (0,1)$ and $r_2=\frac{2}{\theta_4}=\frac{\frac{2}{n}+k_1-1}{\frac{k_1    }{2}-\frac{1}{2}}>0$ by $k_1>1$. Let 
\begin{align*}
y(t)= \frac{1}{{k_2}}\int_{\Omega} (1+u)^{k_2} dx 
 + \frac{1}{{k_1}} \int_{\Omega} w^{k_1} dx
\end{align*}
and $r=\min\{\frac{r_1(k_2+p)}{2k_2}, \frac{r_2}{2}\}>0$. Thus, according to Young's inequality, along with (\ref{u-G-N}) and (\ref{w-G-N}), we can find positive constants $c_{12}=c_{12}(k_1,k_2)$, $c_{13}=c_{13}(k_1,k_2)$, $c_{14}=c_{14}(k_1,k_2)$, $c_{15}=c_{15}(k_1,k_2)$ and $c_{16}=c_{16}(k_1,k_2)$ such that
\begin{align*}
c_{12}y^{r}(t) 
\leq& c_{12}\Big(\frac{2}{{k_2}}\Big)^{r} \left(\int_{\Omega} (1+u)^{k_2} dx \right)^r
 + c_{12}\Big(\frac{2}{{k_1}}\Big)^{r} \Big(\int_{\Omega} w^{k_1} dx \Big)^r \nonumber \\
\leq& c_{13}\Big(\int_{\Omega} (1+u)^{k_2} dx \Big)^{\frac{r_1(k_2+p)}{2k_2}}
 + c_{14} \Big(\int_{\Omega} w^{k_1} dx \Big)^{\frac{r_2}{2}}+c_{15} \nonumber \\
 \leq &  \frac{2K_D({k_2}-1)}{({k_2}+p)^2} \int_{\Omega} |\nabla (1+u)^\frac{{k_2}+p}{2}|^2dx
+\frac{2({k_1}-1)}{{k_1}^2} \int_{\Omega} |\nabla w^{\frac{{k_1}}{2}}|^2dx +c_{16}.
\end{align*}
Thus, we obtain
\begin{align*}
y^{\prime}(t)+c_{12}y^{r}(t) \leq c_{17}, \quad t \in (0,T_{\max})
\end{align*}
with $c_{17}=c_{17}(k_1,k_2)=2c_{8}+2c_{4}+c_{16}$ which implies that $y(t) \leq \min\{y(0), (\frac{c_{17}}{c_{12}})^{-r}\}$. Since $\frac{n}{2}\big(k_1+\frac{2}{n}\big)-\frac{2}{n}(p-q+2)-p\rightarrow \infty$ and $\big(\frac{2}{n}(p-q+2)+p-q+1\big)\big(k_1+\frac{2}{n}\big)-\frac{2}{n}(p-q+2)-p\rightarrow \infty$ as $k_1\rightarrow \infty$ in (\ref{k2}), we can prove (\ref{uw-bounded}) for all $k_1,k_2>1$ by Hölder's inequality.
\end{proof}
A standard argument relying on a Moser-type recursion finally yields $L^{\infty}$ bounds:

\emph{\textbf{Proof of Theorem \ref{thm1_2}.}} We further use a Moser-type iteration (cf. \cite[Lemma A.1]{2012-JDE-TaoWinkler}) to obtain the $\|u\|_{L^{\infty}} \leq c_1$ and $\|w\|_{L^{\infty}} \leq c_1$ with $c_1>0$ independent of $t$. Thus, Theorem~\ref{thm1_2} follows from Proposition~\ref{local} .
\section{Global existence. Proof of Theorem \ref{thm1_3}}\label{global existence}
In this section, we aim to prove Theorem \ref{thm1_3}. 
\begin{lem}\label{global}
Let $n \geq 3$. Assume that $D(s)$ and $S(s)$ satisfy $(\ref{eq1.2})$ as well as $(\ref{D-g})$ and $(\ref{S-g})$ with some $K_D, K_S>0$ and $p,q\in \mathbb{R}$ fulfilling 
\begin{align*}
q< 1-\frac{n}{2}.
\end{align*}
Then, for any $T \in (0,T_{\max})$ and $k_1, k_2\in(1,\infty)$, there exists a positive constant $C(k_1, k_2, T)$ such that
\begin{align}\label{uw}
  \|u(\cdot, t)\|_{L^{k_1}\left(\Omega \right)}+\|w(\cdot, t)\|_{L^{k_2}\left(\Omega \right)} \leq C(k_1, k_2, T) , 
  \quad  t\in (0,T).
\end{align}
\end{lem}
\begin{proof}
Since $q< 1-\frac{n}{2}$, for any $k_1>1$, we infer that
\begin{align*}
\frac{n}{2}k_1+1< (k_1+\frac{2}{n})(1-q),
\end{align*}
thus guaranteeing that we can choose $k_2>1$ such that 
\begin{align}\label{k1k2}
\frac{n}{2}k_1+1 < k_2 < (k_1+\frac{2}{n})(1-q).
\end{align}
Therefore, we have 
\begin{align}\label{q}
\frac{k_1k_2}{k_2-1}< k_1+\frac{2}{n}, \quad 
\frac{k_2}{1-q}< k_1+\frac{2}{n}.
\end{align}

From (\ref{uk1}) and (\ref{wk2}), we have
\begin{align}\label{ineq}
& \frac{d}{d t} \left(\frac{1}{{k_2}}\int_{\Omega} (1+u)^{k_2} dx
+\frac{1}{{k_1}} \int_{\Omega} w^{k_1} dx \right)
+\frac{4K_D({k_2}-1)}{({k_2}+p)^2} \int_{\Omega} |\nabla (1+u)^\frac{{k_2}+p}{2}|^2dx \nonumber \\
&  + \frac{4({k_1}-1)}{{k_1}^2} \int_{\Omega} |\nabla w^{\frac{{k_1}}{2}}|^2dx \nonumber \\  
 \leq & \frac{K_S({k_2}-1)}{{k_2}+q-1} \int_{\Omega}(1+u)^{{k_2}+q-1} wdx 
+\frac{{k_1}-1}{{k_1}}  \int_{\Omega} w^{k_1} (1+u) dx.
\end{align}
The Gagliardo-Nirenberg inequality implies the existence of  positive constants $c_1=c_1(k_1)$, $c_2=c_2(k_1)$ and $c_3=c_3(k_1)$ such that
\begin{align}\label{G-N-4}
c_1\int_{\Omega} w^{k_1+\frac{2}{n}} dx
\leq& c_2 \left(\int_{\Omega} w dx\right)^{(k_1+\frac{2}{n})(1-\theta_5)}
\int_{\Omega} |\nabla w^{\frac{{k_1}}{2}}|^2dx
+c_2 \left(\int_{\Omega} w dx\right)^{k_1+\frac{2}{n}} \nonumber \\
\leq& \frac{2({k_1}-1)}{{k_1}^2}\int_{\Omega} |\nabla w^{\frac{{k_1}}{2}}|^2dx +c_3
\end{align}
with $\theta_5=\frac{\frac{k_1}{2}-\frac{k_1}{2(k_1+\frac{2}{n})}}{\frac{k_1}{2}+\frac{1}{n}-\frac{1}{2}} \in (0,1)$ by $k_1>1$. By Young's inequality and (\ref{G-N-4}), along with the second inequality in (\ref{q}), we know that there exist positive constants $c_4=c_4(k_2)$, $c_5=c_5(k_2)$, $c_6=c_6(k_1,k_2)$, $c_7=c_7(k_1,k_2)$ such that
\begin{align}\label{uk1+q-1w-1}
\frac{K_S({k_2}-1)}{{k_2}+q-1} \int_{\Omega}(1+u)^{{k_2}+q-1} wdx 
\leq& c_4 \int_{\Omega} (1+u)^{k_2} dx 
+ c_5 \int_{\Omega} w^{\frac{k_2}{1-q}} dx \nonumber \\
\leq& c_4 \int_{\Omega} (1+u)^{k_2} dx 
+ c_1 \int_{\Omega} w^{k_1+\frac{2}{n}} dx +c_6 \nonumber \\
\leq& c_4 \int_{\Omega} (1+u)^{k_2} dx
+\frac{2({k_1}-1)}{{k_1}^2}\int_{\Omega} |\nabla w^{\frac{{k_1}}{2}}|^2dx +c_7.
\end{align}
Likewise, using the first inequality in (\ref{q}), we can find positive constants $c_8=c_8(k_1,k_2)$, $c_9=c_9(k_1,k_2)$, $c_{10}=c_{10}(k_1,k_2)$ and $c_{11}=c_{11}(k_1,k_2)$ such that
\begin{align}\label{wk2u-1}
\frac{{k_1}-1}{{k_1}}  \int_{\Omega} w^{k_1} (1+u) dx
\leq& c_8 \int_{\Omega} (1+u)^{k_2} dx 
+ c_9 \int_{\Omega} w^{\frac{k_1k_2}{k_2-1}}dx \nonumber \\
\leq& c_8 \int_{\Omega} (1+u)^{k_2} dx 
+ c_1 \int_{\Omega} w^{k_1+\frac{2}{n}}dx +c_{10} \nonumber \\
\leq & c_8 \int_{\Omega} (1+u)^{k_2} dx 
+\frac{2({k_1}-1)}{{k_1}^2}\int_{\Omega} |\nabla w^{\frac{{k_1}}{2}}|^2dx +c_{11}.
\end{align}
Substituting (\ref{uk1+q-1w-1}) and (\ref{wk2u-1}) into (\ref{ineq}) yields 
\begin{align*}
& \frac{d}{d t} \left(\frac{1}{{k_2}}\int_{\Omega} (1+u)^{k_2} dx
+\frac{1}{{k_1}} \int_{\Omega} w^{k_1} dx+\frac{c_7+c_{11}}{k_2(c_4+c_8)} \right) \nonumber \\
\leq &  (c_4+c_8)\int_{\Omega} (1+u)^{k_2} dx+c_7+c_{11} \nonumber \\ 
\leq &(c_4+c_8)k_2\left(\frac{1}{{k_2}}\int_{\Omega} (1+u)^{k_2} dx
+\frac{1}{{k_1}} \int_{\Omega} w^{k_1} dx+\frac{c_7+c_{11}}{k_2(c_4+c_8)}\right), 
\end{align*}
for all $ t\in (0,T_{\max})$. Let
\begin{align*}
y(t)=\frac{1}{{k_2}}\int_{\Omega} (1+u)^{k_2} dx
+\frac{1}{{k_1}} \int_{\Omega} w^{k_1} dx+\frac{c_7+c_{11}}{k_2(c_4+c_8)},
\quad t\in (0,T_{\max}).
\end{align*}
Then, we have
\begin{align*}
y^{\prime}(t) \leq c_{12} y(t), \quad t\in (0,T_{\max})
\end{align*}
with $c_{12}=c_{12}(k_1,k_2)=(c_4+c_8)k_2$. By Grönwall’s lemma, for any $T \in (0,T_{\max})$, we obtain
\begin{align*}
y(t) \leq \mathrm{e}^{c_{12}t}y(0) \leq \mathrm{e}^{c_{12}T}y(0) , \quad t\in (0,T).
\end{align*}
Since $\frac{n}{2}k_1+1\rightarrow \infty $ and $(k_1+\frac{2}{n})(1-q)\rightarrow \infty$ as $k_1\rightarrow \infty$ in (\ref{k1k2}), we can apply Hölder's inequality to obatin (\ref{uw}).
\end{proof}

\emph{\textbf{Proof of Theorem \ref{thm1_3}.}} In view of Lemma A.1 in \cite{2012-JDE-TaoWinkler} and Lemma~\ref{global}, we obtain that if $T_{\max}<\infty$, then one can find $c_1>0$ such that $\|u\|_{L^{\infty}}+\|w\|_{L^{\infty}} \leq c_1$ for all $t \in (0,T_{\max})$. Since this would contradict the extensibility criterion (\ref{extensibility}) in Lemma~\ref{local}, we thus arrive at the conclusion that the system (\ref{eq1.1.0}) has a global classical solution.



\end{document}